\documentclass[noinfoline]{imsart}

\usepackage{amsfonts, amssymb,amsthm, amsmath, enumerate, bm,xspace}
\usepackage[linktoc=page, colorlinks, linkcolor=blue, citecolor=blue]{hyperref}
\usepackage{pbox}

\usepackage{epsfig,subfigure}
\graphicspath{{figures/}}
\usepackage{epstopdf}
\usepackage{caption}
\usepackage{floatrow,pdfsync} %
\usepackage[dvipsnames]{xcolor}

\DeclareMathOperator{\E}{\mathbb{E}}

\DeclareMathOperator*{\diag}{diag}

\renewcommand{\Pr}[2][]{\mathbb{P}_{#1} \left\{ #2 \rule{0mm}{3mm}\right\}}
\newcommand{\ip}[2]{\langle#1,#2\rangle}

\def \N {\mathbb{N}}
\def \P {\mathbb{P}}
\def \R {\mathbb{R}}

\def \LL {\mathcal{L}}
\def \MM {\mathcal{M}}

\def \BB {\mathcal{B}}
\def \ZZ {\mathcal{Z}}

\def \a {\alpha}
\def \b {\beta}

\def \e {\varepsilon}
\def \d {\delta}
\def \l {\lambda}

\def \tran {\mathsf{T}}

\def \onevector {{\bf 1}}

\DeclareMathOperator*{\argmax}{arg\,\! max}

\newtheorem{theorem}{Theorem}[section]

\newtheorem{corollary}[theorem]{Corollary}
\newtheorem{lemma}[theorem]{Lemma}

\newtheorem{remark}[theorem]{Remark}

\startlocaldefs

\endlocaldefs

\begin{document}

\begin{frontmatter}

% "Title of the Paper"
\title{Estimating a network from multiple noisy realizations}
%\thankstext{t1}{This is an original survey paper}
\runtitle{Estimating a network from multiple noisy realizations}

% indicate corresponding author with \corref{}
% \author{\fnms{John} \snm{Smith}\thanksref{t2}\corref{}\ead[label=e1]{smith@foo.com}\ead[label=e2,url]{www.foo.com}}
% \thankstext{t2}{Thanks to somebody} 
% \address{line 1\\ line 2\\ \printead{e1}\\ \printead{e2}}

\author{\fnms{Can M.} \snm{Le}\ead[label=e1]{canle@ucdavis.edu}}
\address{Department of Statistics, University of California Davis, Davis, CA 95616, USA\\
\printead{e1}}
%\and
\author{\fnms{Keith} \snm{Levin}\ead[label=e2]{klevin@umich.edu}}
\address{Department of Statistics, University of Michigan, Ann Arbor, MI 48109, USA \\\printead{e2}}
 % \and
\author{\fnms{Elizaveta} \snm{Levina}\ead[label=e3]{elevina@umich.edu}}
\address{Department of Statistics, University of Michigan, Ann Arbor, MI 48109, USA \\\printead{e3}}

\runauthor{Le et al.}

\begin{abstract}
Complex interactions between entities are often represented as edges in a network.    In practice, the network is often constructed from noisy measurements and inevitably contains some errors.
In this paper we consider the problem of estimating a network from multiple noisy observations where edges of the original network are recorded with both false positives and false negatives.  This problem is motivated by neuroimaging applications where brain networks of a group of patients with a particular brain condition could be viewed as noisy versions of an unobserved true network corresponding to the disease. 
The key to optimally leveraging these multiple observations is to take advantage of network structure, and here we focus on the case where the true network contains communities.  Communities are common in real networks in general and in particular are believed to be presented in brain networks.    
Under a community structure assumption on the truth, we derive an efficient method to estimate the noise levels and  the original network, with theoretical guarantees on the convergence of our estimates.   We show on synthetic networks that the performance of our method is close to an oracle method using the true parameter values, and apply our method to fMRI brain data, demonstrating that it constructs stable and plausible estimates of the population network.  
\end{abstract}



% history:
% \received{\smonth{1} \syear{0000}}

\tableofcontents

\end{frontmatter}

\section{Introduction}
Networks provide a natural way to model many complex systems, and network data are increasingly common in many areas of application.   Statistical network analysis to date has largely focused on the case of observing a single network, without noise, and analyzing the observed network in order to learn something about its structure, for example, identifying communities.   The problem of community detection in particular,  in a single noiseless network, is very well studied and understood by now  (see \cite{Goldenberg2010,Fortunato2010,AbbeReview2017} for reviews of this topic and \cite{Amini.et.al.2013,Chin&Rao&Vu2015,Le&Levina&Vershynin2017,Gao&Ma&Zhang&Zhou2015} for some of the many important recent developments). Much effort in this field has focused on the analysis of exchangeable networks, where any permutation of nodes results in the same distribution of the edges \cite{Bickel&Chen2009,Karrer10,Choi&Wolfe2014,Olhede&Wolfe2014}.

In this paper, our focus is on applications where multiple noisy realizations are available rather than a single network, much like an i.i.d.\ sample in classical multivariate analysis, except our observations are networks rather than vectors. 
The particular application that motivated this work is neuroimaging, where a network of connections in the brain is constructed separately for each subject, and there is a sample of subjects available, e.g., people suffering from a mental illness.   Nodes in this context correspond to locations or regions of interest in the brain, and connections between nodes are measured in various ways depending on the technology used.  Here we focus on data from resting state fMRI brain imaging \cite{Taylor&Chen&et.al.2011,Taylor&Demeter&et.al.2014}, where time series of blood oxygen levels are recorded at multiple voxels in the brain while the subjects ``rest'' in the fMRI machine (see Section~\ref{subsec: brain network} for more details).  Inferring connections between nodes from this type of data invariably involves a lot of preprocessing (registration,  background subtraction, normalization, etc.), and is typically measured by computing Pearson correlations between the processed time series for each pair of nodes, although arguments have also been made for using partial correlations and more generally Markov random fields \cite{Narayan&Allen&Tomson2015,Narayan&Allen2016}.
%\liza{add references to Genevera Allen}.   

%\liza{\cite{Rubinov&Sporns2010,Power&Cohen&et.al.2011,}.   This is a mix of general references and references specific to the data set.  Need to separate them and elaborate.}
%In this context, subjects are not given any task and blood oxygen levels are recorded at multiple locations in the brain over a period of time, resulting in a time series recorded at every voxel.  After preprocessing, which includes registration,correlations between spatial locations are computed by averaging over the time dimension and

However the connections between nodes are computed, they are then frequently thresholded in order to obtain a connectivity matrix with binary entries, from which various network summaries such as the average degree and the clustering coefficient can be computed and averaged over the sample to characterize the population \cite{Buckner.et.al.2009,Bassett.et.al.2012,Hosseini&Kesler2013,Klimm.et.al.2014}. These one-number summaries necessarily result in loss of information, and one may want to learn more about the prototypical brain network for a population of patients beyond one-number summaries. For instance, one may want to find regions of the brain consisting of similar voxels in terms of functional connectivity and comparing them to healthy controls, or compare levels of functional connectivity within known anatomical regions.   A natural question to ask then is how to estimate a population network adjacency matrix $A$ (an $n\times n$ matrix
where $A_{ij}=1$ if there is an edge between node $i$ and node $j$, and $A_{ij}=0$ otherwise) from a sample of noisy observations, with noise resulting from both preprocessing and natural individual variations.    In other words, we pose the question of how to compute the ``mean'' from a sample of $N$ independent noisy realizations $A^{(1)}, A^{(2)},\dots,A^{(N)}$ of an unknown underlying adjacency matrix $A$, while respecting and ideally taking advantage of the network structure of the problem instead of simply averaging the observed matrices. 
%Note that this problem is distinct from the problem of estimating the edge probability matrix studied in \cite{Tang&Ketcha&Vogelstein&Priebe&Sussman2016}.

We next introduce basic notation to focus the discussion.  Since the underlying true $A$ is binary and so are the observations, the noise in each entry of $A$ can only be present in the form of false positive and false negative edges.   We assume that the entries of $A$ above the diagonal are generated independently (an assumption that certainly simplifies reality but enables analysis that has been found to give useful practical results in much of previous literature on networks), and that the noise is independent of $A$.
Let $P$ be the $n \times n$ symmetric matrix of false positive probabilities, and $Q$ the $n\times n$ symmetric matrix of false negative probabilities.
That is, for each $1 \le m \le N$ and $i<j$,
if $A_{ij}=1$ then $A^{(m)}_{ij}$ is drawn from $\mathrm{Bernoulli} (1-Q_{ij})$,
and if $A_{ij}=0$ then $A^{(m)}_{ij}$ is drawn from $\mathrm{Bernoulli} (P_{ij})$.
The entries above the diagonal of $A^{(m)}$  are independent, $A^{(m)}_{ij} = A^{(m)}_{ji}$,  and diagonal entries of $A^{(m)}$ are set to zero, though the latter is not important.
In other words, true edges $A_{ij} = 1$ are randomly removed with probabilities $Q_{ij}$
while non-edges $A_{ij} = 0$ are randomly replaced with false edges with probabilities $P_{ij}$.   For identifiability, we assume that all entries of $P$ and $Q$ are less than $1/2$.

In principle, each entry $A_{ij}$ of the underlying true network $A$ can be estimated separately from the corresponding entries $A_{ij}^{(m)}$, $1\le m \le N$.
However, this naive approach does not take advantage of any potential structure in $A$. Given that real networks typically exhibit a lot of structure, we can expect to gain by estimating the entries of $A$ jointly.    More specifically, we assume that the structure in $A$ takes the form of communities, frequently encountered in many real-world networks in general and in brain networks in particular
\cite{Power&Cohen&et.al.2011}.    We will model this structure in $A$ through one of the most commonly used network community models, the stochastic block model (SBM) \cite{Holland83}.   The SBM is a simple and easily tractable model which can also serve  as a basic building block in approximating a much larger family of network models, much in the same way that a step-wise constant function can be used to approximate any smooth function \cite{Olhede&Wolfe2014}.   Making this assumption about $A$ allows us to share information among edges while retaining the flexibility to fit a wide range of network data.

The SBM assumes that the network is generated by first drawing a vector of node labels
$c\in\{1,\dots,K\}^n$ from a multinomial distribution with parameter $\pi = \{\pi_1, \dots, \pi_K\}$.
The number of communities $K$ is often assumed to be known, or can be estimated by using one of several methods now available \cite{Chen&Lei2014,Wang&Bickel2015,Le&Levina2015}.
Edges between pairs of nodes $i,j$ are then drawn independently with probability $P(A_{ij} = 1) = B_{c_i c_j}$,
where $B$ is a $K\times K$ matrix of within and between communities edge probabilities.
Following the literature, we condition on $c$ and treat it as a fixed unknown vector from this point on.   Community detection under the SBM has been studied intensively in the last decade and many methods are available by now, e.g.,
\cite{NewmanPNAS,Bickel&Chen2009,Amini.et.al.2013,
Krzakala.et.al2013spectral,Le&Levina&Vershynin2017}, and many others.
%\liza{13 references without any description is an overkill.   Reduce to 4-5 most important/relevant ones (also cite in the first paragraph), and add ``and many others''.}

We make a further assumption that the expectation $W=\E A$ of $A$ and the noise probability matrices $P$ and $Q$ share the same block structure.
That is, if $c_i=c_{i^\prime}$ and $c_j=c_{j^\prime}$ then
$P_{ij} = P_{i^\prime j^\prime}$ and $Q_{ij} = Q_{i^\prime j^\prime}$.
In other words, edges between nodes with the same patterns of connectivity are subject to the same noise levels. For the SBM, one can think of this assumption as the probability of making an error about an edge being a function of the probability of that  edge existing.  In many biological contexts, it is plausible to assume that the probability of a false negative is higher when the probability of an edge is small, as it is harder to detect, and conversely for an edge with high probability, the probability of a false negative might be low.

The main contribution of this work is an algorithm to estimate the true unobserved ``population'' adjacency matrix  $A$ by taking advantage of the community structure in both the network and the noise.  The algorithm works by first estimating the community structure of $A$ from an initial naive estimate, using an existing method
such as spectral clustering %\cite{Amini.et.al.2013,Joseph&Yu2013,Le&Levina&Vershynin2015}
or pseudo-likelihood  \cite{Amini.et.al.2013}.  Then the estimated community structure is used in an EM-type algorithm to update the estimate of $A$ and the parameters of interest.    Results in Section~\ref{sec: numerical results} show that our method performs well on both simulated data
and functional connectomics brain data  \cite{Taylor&Chen&et.al.2011,Taylor&Demeter&et.al.2014}.
The method is computationally efficient because we can leverage existing fast algorithms for community detection in the first stage and the EM algorithm in the second stage only involves simple updates which converge quickly.  More complicated models of the relationship between the network and the noise are certainly possible and are left to future work, but even with this simple model we demonstrate conclusively that ``network-aware'' analyses of samples of networks, as opposed to ``massively univariate'' analyses that vectorize the adjacency matrices and ignore their network structure, are needed to take full advantage of the network nature of the data.

%In network literature, i.i.d. network observations have been considered in the context of graphon estimation \cite{Airoldi&Ketcha&Costa&Chan2013} and edge probability estimation \cite{Tang&Ketcha&Vogelstein&Priebe&Sussman2016}. 
The problem we consider in this paper shares some similarity with the problem of estimating the edge probability matrix from independent network observations $A^{(1)},...,A^{(N)}$ studied in \cite{Tang&Ketcha&Vogelstein&Priebe&Sussman2016,Wang&Zhang&Dunson2017}. 
Assuming that $A^{(1)},...,A^{(N)}$ are identically distributed and $\E A^{(1)}$ is of low rank, the authors of \cite{Tang&Ketcha&Vogelstein&Priebe&Sussman2016} estimate $\E A^{(1)}$ by a low rank approximation $H$ of $N^{-1} \sum_{m=1}^N A^{(m)}$. In \cite{Wang&Zhang&Dunson2017}, the authors model the entrywise logit of $\E A^{(m)}$ as the sum of a baseline matrix $Z$ and an individual-specific matrix $D_m$ and propose a spectral method to estimate them. Note that in our setting, $\E A^{(1)}$ is a matrix with entries $\E A_{ij}^{(1)}=P_{ij}$ if $A_{ij} = 1$ and $\E A_{ij}^{(1)}=1-Q_{ij}$ if $A_{ij} = 0$. Since entries of $P$ and $Q$ are less than $1/2$, in principle one can threshold entries of $H$ or the estimate of $Z$ at $1/2$ to obtain an estimate of $A$. However, these are not good estimates because (i) $\E A^{(1)}$ is not a low-rank matrix, (ii) they are not designed specifically for estimating a binary matrix, and (iii) estimates of $P$ and $Q$ are required for a noise-dependent threshold. Therefore the problem of estimating a binary network must be treated differently, and it is the main focus of this work.

Finally, there is a connection between the problem we study in this paper and the problem of crowdsourcing  \cite{Dawid&Skene1979}.    Crowdsourcing aims to recover the latent labels of a set of items based on independent estimates of several workers; in our setting, (binary) $A_{ij}$ is the latent label of the item indexed by $(i,j)$ and $A^{(m)}_{ij}$ is an estimate of the $m$-th worker. A number of methods have been developed to address this problem, including SVD-based methods \cite{Ghosh2011WhoMT}, variational methods \cite{Liu&Peng&Ihler2012}, Bayesian inference \cite{Raykar.et.al.2010} and EM algorithms \cite{Dawid&Skene1979,Zhang&Chen&Zhou&Jordan2014}.    The two-stage procedure of \cite{Zhang&Chen&Zhou&Jordan2014} is especially relevant to our paper, where the labels are initialized by the method of moments and updated by the EM algorithm. % It has been shown in \cite{Zhang&Chen&Zhou&Jordan2014} that the procedure is efficient and achieves the minimax convergence rate up to a log factor.

Our setting corresponds to crowdsourcing if  we take the number of communities to be $K=1$, and ignore any network structure  in particular the fact that the $n(n-1)/2$ edge labels come from only $n$ nodes, and that these $n$ nodes form communities.    The setting with a general  community structure is much more challenging, because it requires estimating two layers of latent variables, the community labels and the edge values themselves.  Taking community structure into account is crucial for the method to be relevant in neuroimaging applications, and differs from the crowdsourcing setting in highly non-trivial ways.

\section{Optimal estimates and the role of noise}\label{sec:known parameters}
We start by deriving two estimators of $A$ when parameters $W,P,Q$ are known:
a maximum likelihood estimator and an estimator based on likelihood ratio tests.
These are not practical, but since they are provide optimal estimation error and test power, it is instructive
to understand their behavior as a function of noise level.  We will also
use these estimators as oracle benchmarks for comparisons, and
to derive the EM algorithm presented in the next section.
%It will also be used as benchmarks to compare with other estimates.

When $W,P$ and $Q$ are known and the only unknown is the underlying matrix
$A$, treated as fixed, we can estimate each entry $A_{ij}$
independently, since the only source of randomness is independent
noise.
To simplify notation, we fix a pair $(i,j)$ of nodes and denote $a = A_{ij}$, $a_m = A_{ij}^{(m)}$, $s=\sum_{m=1}^N a_m$, $w = W_{ij}$, $p = P_{ij}$ and $q = Q_{ij}$.

\subsection{Maximum likelihood estimation}\label{sec: MLE optimal}
The likelihood of $a$ given the data $a_1,...,a_N$ is
\begin{equation*}
   \LL(a; a_1, \dots, a_N) =
   \Big[ w \prod_{m=1}^N
   (1-q)^{a_m} q^{1-a_m} \Big]^{a} \cdot
   \Big[ (1-w) \prod_{m=1}^N
   p^{a_m} (1-p)^{1-a_m} \Big]^{1-a}.
\end{equation*}
Up to a constant, we can write the log-likelihood as
\begin{equation}
   \log \LL(a) \propto
   a \left(
   s \log \frac{(1-p)(1-q)}{pq}-
   \log \frac{1-w}{w}-
   N \log \frac{1-p}{q} \right).
\label{eq:loglike}
\end{equation}
Since $a$ can only
take on values of 0 or 1, the estimate will be determined by the sign
of the multiplier of $a$ in \eqref{eq:loglike}.
Therefore, the maximum likelihood estimator of $a$ is
\begin{equation}\label{eq: maximum likelihood estimator, known parameters}
  a^* =\mathbf{1}{ \{s \ge \mu \} },
  \quad \mathrm{ where } \quad
  \mu = \frac{ \log \frac{1-w}{w} + N \log \frac{1-p}{q}}
  {\log \frac{(1-p)(1-q)}{pq}}.
\end{equation}
%Let $A^*$ denote the matrix with entries $A_{ij}^*$.
To understand how the optimal estimate $a^*$ depends on the noise,
consider the estimation error of $a^*$, which has the form
\begin{equation}\label{eq: estimation error}
  \P (a^* \neq a) =
  w \ \P (s < \mu| a =1 )
  + (1 - w) \ \P ( s \ge \mu | a = 0).
\end{equation}
The probabilities are binomial:  conditional on $a=1$, $s$ is
  $\operatorname{Binomial}(N, 1-q)$, and conditional on $a = 0$,  $s$ is
  $\operatorname{Binomial}(N, p)$.  Since the threshold $t$ depends on
  $p$, $q$, and $w$, the dependence of the
  error on these parameters is somewhat complicated, but
  straightforward to compute.
Figure~\ref{Fig:MLErrorContour} shows the error $\P (a^* \neq a)$
as a function of $p$ and $q$.
When $p$ or $q$ increases and all other parameters are fixed, the estimation error increases.
This observation is confirmed by the following lemma (the proof is given in Appendix~\ref{ap: estimation error}).

\begin{lemma}[The role of noise]\label{lem: estimation error}
The estimation error $\P(a^*\neq a)$ defined by \eqref{eq: estimation error} is an increasing function of $p$ and $q$.
\end{lemma}

%\liza{I find this section a bit unsatisfactory, and it ends abruptly.
%  Any other insights to add about the relationship between P and Q and
%  W?   Is it true, for example, that the error is increasing in P and
%  in Q when the rest is fixed?  It should be.  This can be checked
%  simply by taking derivatives.   It's also confusing
% because the entire section is for a single entry, and yet the SBM has
% already been mentioned and assumptions made about the block structure of P
% and Q matching that of A.... Maybe change notation so that this is
% all without the subscript $ij$, or will it look too weird?
%}

\begin{figure}
  \centering
  \includegraphics[trim=130 40 100 20,clip,width=\textwidth]{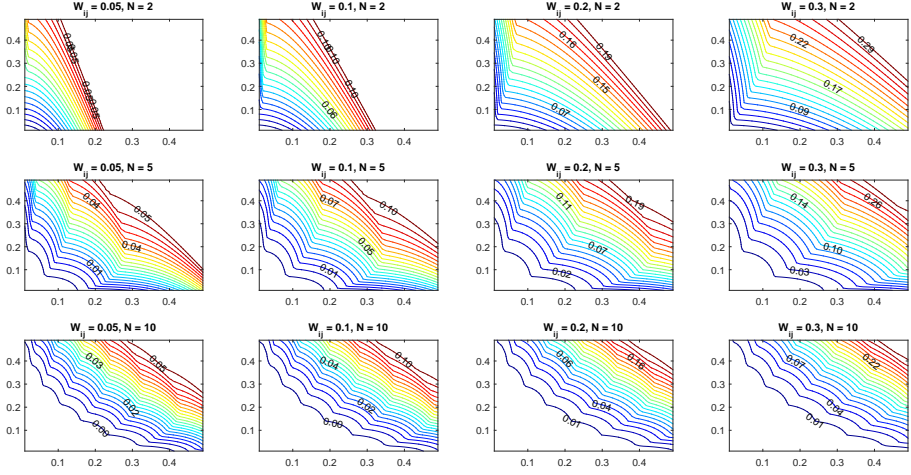}\\
  \caption{The contour of the estimation error of $A_{ij}^*$ as a function of
  $P_{ij}$ on the $x$-axis and $Q_{ij}$ on the $y$-axis.
  %On each plot, the contour value increases from zero to $M_{i,j}$.
  The errors shown on each plot are measured at coordinates
  $(0.4,0.4)$, $(0.2,0.4)$, $(0.3,0.2)$ and $(0.2,0.1)$.}
  \label{Fig:MLErrorContour}
\end{figure}

\subsection{Likelihood ratio tests and FDR}\label{sec: likelihood ratio test}
An alternative approach to estimating $a$ when all parameters are known is to perform a test.
Unlike maximum likelihood estimation, testing allows us to explicitly control the false discovery rate (FDR), which is often important in practice.

Consider the null hypothesis $a=0$ and the alternative hypothesis $a=1$.
Under the null, $s = \sum_{m=1}^N a_m$ follows $\operatorname{Binomial}(N,p)$;
under the alternative, $s$ follows $\operatorname{Binomial}(N,1-q)$.
For a given confidence level $\a$, let $T_\a$ be a likelihood ratio test with critical value $k_\a\in\N$ that accepts the null if $s < k_\a$ and accepts the alternative if $s> k_\a$; when $s=k_\a$, it accepts the alternative with a certain probability adjusted to achieve the level $\a$. The power $\gamma_\a$ of the test $T_\a$ is then also a function of $\a$.
Since
\begin{eqnarray*}
\Pr{T_\a\text{ rejects the null}} &=& \a\Pr{a=0} + \gamma_\a\Pr{a=1} = \a(1-w)+\gamma_\a w,\\
\Pr{T_\a\text{ falsely rejects the null}} &=& \a\Pr{a=0}= \a(1-w),
\end{eqnarray*} 
the false discovery rate $\xi_\a$ of $T_\a$ can be computed by
\begin{equation}\label{eq: FDR}
  \xi_\a = \frac{\a(1-w)}{\a(1-w) + \gamma_\a w}.
\end{equation}
%where $\gamma$ is the power of $T_\a$ (see e.g. \cite{Storey2002} for a discussion on this topic).
We state a property of $\xi_\a$ that we will use to control the false discovery rate (the proof is given in Appendix~\ref{ap: estimation error}). 

\begin{lemma}[Monotonicity of false discovery rate]\label{lem: monotonicity of fdr}
Consider the likelihood ratio test $T_\a$ and the false discovery rate $\xi_\a$ of $T_\a$ defined by \eqref{eq: FDR}.
If $p,q\le 1/2$ then $\xi_\a$ is an increasing function of $\a$.
\end{lemma} 

It is easy to see that $\xi_\a$ takes values from zero to $1-w$ ($\gamma_\a$ tends to one as $\a$ tends to one). Therefore, for a fixed $\xi\in(0,1-w)$, by the monotonicity of $\xi_\a$, we can estimate $a$ by performing test $T_\a$ with $\a$ being the unique solution of equation \eqref{eq: FDR}.

Figure~\ref{fig: test power}
shows the power $\gamma_\a$ of $T_\a$ as a function of $p$ and $q$
when $N=10$, $w=0.2$ and the false discovery rate is fixed at $\xi_\a=0.05$. We see that $\gamma_\a$ decreases when either $p$ or $q$ increases and the other is fixed. Also, $\gamma_\a$ is close to one when both $p,q$ are small and $\gamma_\a$ is close to zero when both $p,q$ are large.

In Section~\ref{sec: algorithm}, we estimate unknown parameters via the EM algorithm and use the estimates as plugins for the unknown parameters to perform likelihood ratio tests.

\begin{figure}
\begin{center}
\includegraphics[trim=20 0 30 20,clip,width=0.65\textwidth]{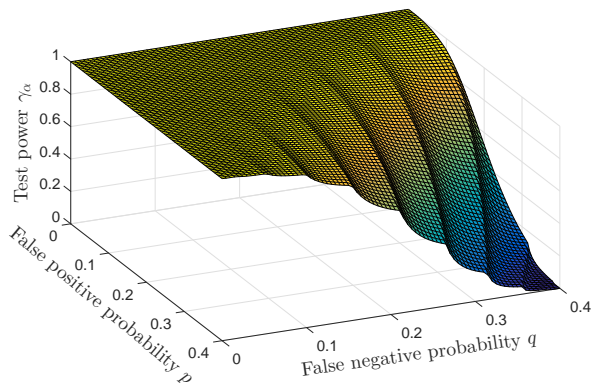}
\end{center}
\caption{Power of the likelihood ratio test $T_\a$ with $w=0.2$, $N=10$ and the false discovery rate fixed at $\xi_\a=0.05$.}
\label{fig: test power}
\end{figure}

\section{The estimation algorithm}\label{sec: algorithm}
We now return to the more realistic case  of unknown parameters, and derive an algorithm to estimate  $A$, $W$, $P$ and $Q$ at the same time.
We obtain an initial estimate of the underlying block structure
shared by $W$, $P$ and $Q$  from the matrix $S = \sum_{m=1}^N
A^{(m)}$.   Then we apply the EM algorithm to estimate submatrices of
$A$, $W$, $P$ and $Q$ associated with the estimated blocks. 

\subsection{Estimating the block structure} \label{sec: initial block estimation}
%Denote by $S$ the sum of observations $A^{(m)}$, $1\le m \le N$,
Let $\hat{A} = (\hat A_{ij})$ be the $n \times n$ matrix with
entries $\hat{A}_{ij} = \mathbf{1}(S_{ij} \ge N/2)$.
Under the assumption that entries of $P$ and $Q$ are at most $1/2$,
without which the problem becomes unidentifiable,
the matrix $\hat{A}$ is a consistent estimate of $A$.
We can estimate the block structure of $W$ by applying a community detection algorithm,
such as spectral clustering
\cite{Amini.et.al.2013,Joseph&Yu2013,Le&Levina&Vershynin2015} or the pseudo-likelihood method  \cite{Amini.et.al.2013},
to the initial estimate $\hat{A}$.
We will assume that the number of communities $K$ is known, as is
usually done in network literature, or alternatively $K$ can be first
estimated by one of several methods available
\cite{Chen&Lei2014,Wang&Bickel2015,Le&Levina2015}.
Having estimated the block structure, we condition on the node labels
and treat them as known,
so that the entries of $W$, $P$, and $Q$ are constant within
each estimated block.
With this assumption in mind, we now present the EM algorithm to recover
the sub-matrix of $A$ corresponding to each estimated block.

\subsection{The EM algorithm when node labels are known} \label{sec: em algorithm}
In this section, we derive an estimate of $A$ using the EM
algorithm, assuming that the vector of node labels $c$ is known.
To get the final estimate of $A$, we will replace $c$ with an estimate from Section~\ref{sec: initial block estimation}.

Recall that $A$ is generated from a block model with $K$ communities.
Therefore $W = \E A$ is a symmetric matrix with $K^2$ blocks (determined by
$c$), with equal entries within each block. To focus on one such block, we fix $k,l\in\{1,...,K\}$ with $k\neq l$ (the case $k=l$ is treated similarly) and consider the $(k,l)$ block according to $c$: 
$$
J = \{(i,j): c_i=k, c_j = l\}.
$$
By assumption of shared block structure, restrictions of $W$, $P$ and $Q$ to $J$ are matrices of constant entries, values of which we denote by
$w$, $p$, and $q$, respectively. Thus, $W_{ij} = w$, $P_{ij}=p$ and $Q_{ij}=q$ for all $(i,j)\in J$. 
The likelihood of $A_J$ and $A_J^{(m)}$  --  restrictions of $A$ and $A^{(m)}$ to $J$ -- takes the form
\begin{equation*}
  \LL =
  \prod_{(i,j) \in J}
  \left[w \prod_{m=1}^N  q^{1-A_{ij}^{(m)}}(1-q)^{A_{ij}^{(m)}}\right]^{A_{ij}}
  \left[(1-w) \prod_{m=1}^N p^{A_{ij}^{(m)}} (1-p)^{1-A_{ij}^{(m)}} \right]^{1-A_{ij}}.
\end{equation*}
%\liza{ I don't see where you are using the fact that it is a
%  non-diagonal block;  what difference does it make? }

For each $0 \le r \le N$, define $I_r = \{ (i,j) \in J: S_{ij} = r \}$.
Adding up the log-likelihoods of the independent $A_J^{(m)}$, $1\le m \le N$, and grouping the terms with $(i,j) \in I_r$, we obtain
\begin{eqnarray*}
  \log \LL &=& \sum_{r=0}^N \sum_{(i,j) \in I_r}
\Big\{  A_{ij} \log w + (1-A_{ij}) \log(1-w)\\
 &+& A_{ij} \big[(N-r)\log q + r \log (1-q) \big]  + (1-A_{ij}) \big[r\log p + (N-r)\log(1-p)\big] \Big\}.
\end{eqnarray*}
%\liza{What is $c$?  It is not explained anywhere and is already used for labels.  }
% For each $(i,j) \in I_r$ let $\tau_r = \E[\tilde{A}_{i,j} | \{\tilde{A}^{(m)}\}_{m=1}^N]$.
For each $(i,j)\in I_r$, define $\tau_r = \E[A_{i,j} | S_{ij}=r]$. Hereafter, we use $|R|$ to denote the cardinality of a set $R$.
Taking the conditional expectation of the log-likelihood given the
data, $ \tilde{\LL} = \mathbb{E}(\log \LL | \{A_J^{(m)}\}_{m=1}^N)$,
we obtain for the E-step
\begin{eqnarray*}
  \tilde{\LL} = \sum_{r=0}^N |I_r|
  &\Big\{&  \big[  \tau_r \log w +  (1-\tau_r) \log(1-w) \big] \\
  &+& \tau_r \big[ (N-r) \log q + r \log(1-q) \big]  \\
  &+& (1-\tau_r) \big[ r \log p + (N-r) \log (1-p)\big] \Big\}.
\end{eqnarray*}
The M-step involves finding estimates of $w,p,q$
that maximize $\tilde{\LL}$.
These estimates are unique, since $\tilde{\LL}$ is concave in $w,p,q$.
The partial derivative of $\tilde{\LL}$ with respect to $w$ has the form
\begin{eqnarray*}
  \frac{\partial\tilde{\LL}}{\partial w}
  &=& \sum_{r=0}^N |I_r| \left( \frac{\tau_r}{w} - \frac{1 - \tau_r}{1-w}\right).
\end{eqnarray*}
Setting the derivative to zero yields an estimate $\hat{w}$ of $w$:
\begin{equation}\label{eq: estimate of mu}
  \hat{w} = \frac{1}{|J|} \sum_{r=0}^N \tau_r |I_r|.
\end{equation}
Similarly, the estimates of $p$ and $q$ take the form
\begin{equation}\label{eq: estimates of p and q}
  \hat{p} = \frac{\sum_{r=0}^N r(1-\tau_r)|I_r|}{\sum_{r=0}^N N(1-\tau_r) |I_r|}, \quad
  \hat{q} = \frac{\sum_{r=0}^N \tau_r (N-r) |I_r|}{\sum_{r=0}^N \tau_r N |I_r|}.
\end{equation}
Since the $\tau_r$'s are unknown, we initialize by majority vote $\hat{\tau}_r = \mathbf{1}(r\ge N/2)$.  The Bayes rule gives
\begin{eqnarray*}\label{eq: formula of posterior}
  \tau_r &=& \P( A_{ij} = 1 | S_{ij} = r) \\
  &=& \frac{\P ( S_{ij} = r | A_{ij} = 1 ) \ \P (A_{ij}=1)}
  {\P ( S_{ij} = r | A_{ij} = 1 ) \ \P (A_{ij} = 1) +
   \P ( S_{ij} = r | A_{ij} = 0 ) \ \P (A_{ij} = 0)} \\
  &=& \frac{ w (1-q)^r q^{N-r} }
  { w (1-q)^r q^{N-r} + (1-w) p^r (1-p)^{N-r}}.
\end{eqnarray*}
Therefore, once $\hat{w}$, $\hat{p}$ and $\hat{q}$ are computed, we can update $\hat{\tau}_r$ by
\begin{equation}\label{eq: update of the posterior}
  \hat{\tau}_r
  =\frac{ \hat{w} (1-\hat{q})^r \hat{q}^{N-r} }
  { \hat{w} (1-\hat{q})^r \hat{q}^{N-r} + (1-\hat{w}) \hat{p}^r (1-\hat{p})^{N-r}}.
\end{equation}
The EM steps are then iterated until convergence.

\subsection{The complete EM algorithm with unknown labels}\label{subsec: em-type algorithm}
In Section~\ref{sec: em algorithm} we assume that $c$ is known and derive the EM algorithm for estimating $A$.
Since $c$ is unknown in practice, we first compute its estimate $\hat{c}$ using $\hat{A}$ as described in Section~\ref{sec: initial block estimation}.
We then repeat the following steps until convergence:
(i) treat $\hat{c}$ as the ground truth and estimate $A$ by applying
the EM algorithm described in Section~\ref{sec: em algorithm} for each block,
(ii) update $\hat{c}$ using the new estimate of $A$.
%In practice only a few iterations (outer loops) are needed for the algorithm to converge.

%\liza{need to rewrite this jointly in ``algorithm'' form with inner
%  and outer loop - step 0,
%  step 1, return to step 1, etc.  Also need to discuss convergence
%  criteria, convergence guarantees - are there any?   Should be
%  convex? - and practical considerations, such as a small number of
%  iterations of the outer loop and the need to run EM until
%  convergence before updating labels. }

%The EM algorithm is summarized as follows.
%We initialize $\hat{\tau}_r = \mathbf{1}_{\{ r > N /2 \}}$,
%and update $\hat{w}$, $\hat{p}$, $\hat{q}$, and $\hat{\tau}_r$ by repeating $T$ times:
%\begin{enumerate}
%  \item M-step: compute $\hat{w}$, $\hat{p}$, and $\hat{q}$ using
%                \eqref{eq: estimate of mu} and \eqref{eq: estimates of p and q},
%  \item E-step: update $\hat{\tau}_r$ for each $0 \le r \le N$
%                using \eqref{eq: update of the posterior}.
%\end{enumerate}
Recall that $S$ is the sum of observations $A^{(m)}$, $1\le m \le N$,
and $\hat{A}$ is the matrix with entries $\hat{A}_{ij} = \mathbf{1}(S_{ij} \ge N/2)$.
We initialize an estimate $\hat{c}$ of community labels $c$ by applying an existing clustering algorithm on $\hat{A}$. Although we can choose any consistent clustering algorithm, for concreteness, we will use spectral clustering. 
Similar to Section~\ref{sec: em algorithm}, we fix $k,l\in\{1,...,K\}$ and consider the $(k,l)$ block according to $\hat{c}$: 
$$
\hat{J} = \{(i,j): \hat{c}_i=k, \hat{c}_j = l\}.
$$
Within block $\hat{J}$, we estimate entries of $W$, $P$ and $Q$ by $\hat{w}$, $\hat{p}$ and $\hat{q}$, respectively, and compute them as follows.
Denote
$
\hat{I}_r = \{ (i,j)\in \hat{J}: S_{ij} = r \}
$.
Initialize $\hat{\tau}_r = \mathbf{1}(r\ge N/2)$ and repeat $T$ times:
\begin{enumerate}
  \item Compute $\hat{w}$, $\hat{p}$, and $\hat{q}$ by %using \eqref{eq: estimate of mu}
  $$
  \hat{w} = \frac{1}{|\hat{J}|} \sum_{r=0}^N \hat{\tau}_r |\hat{I}_r|,
  $$
  and %using \eqref{eq: estimates of p and q}
  $$
  \hat{p} = \frac{\sum_{r=0}^N r(1-\hat{\tau}_r)|\hat{I}_r|}{\sum_{r=0}^N N(1-\hat{\tau}_r) |\hat{I}_r|}, \quad
  \hat{q} = \frac{\sum_{r=0}^N (N-r)\hat{\tau}_r  |\hat{I}_r|}{\sum_{r=0}^N N \hat{\tau}_r  |\hat{I}_r|}.
  $$
  \item Using current estimates $\hat{w}$, $\hat{p}$, and $\hat{q}$, update the posterior $\hat{\tau}_r$ %using \eqref{eq: update of the posterior}
  $$
  \hat{\tau}_r = \frac{ \hat{w} (1-\hat{q})^r \hat{q}^{N-r} }
  {\hat{w} (1-\hat{q})^r \hat{q}^{N-r} + (1-\hat{w}) \hat{p}^r (1-\hat{p})^{N-r}}.
  $$
  \item Return to step (1) unless the parameter estimates have converged.
  \item Update the $\hat{J}$ block of $\hat{A}$ by $\hat{A}_{ij}=\mathbf{1}\{\hat{\tau}_r\ge 1/2\}$.
  %$$
%  \hat{A}_{ij} = \mathbf{1}{ \{S_{ij} \ge t_{ij} \} },
%  \quad \mathrm{ where } \quad
%  t_{ij} = \frac{ \log \frac{1-\hat{W}_{\hat{c}_i\hat{c}_j}}{\hat{W}_{\hat{c}_i\hat{c}_j}} + N \log \frac{1-\hat{P}_{\hat{c}_i\hat{c}_j}}{\hat{Q}_{\hat{c}_i\hat{c}_j}}}
%  {\log \frac{(1-\hat{P}_{\hat{c}_i\hat{c}_j})(1-\hat{Q}_{\hat{c}_i\hat{c}_j})}{\hat{P}_{\hat{c}_i\hat{c}_j}\hat{Q}_{\hat{c}_i\hat{c}_j}}}.
%  $$
  \item Update the label estimate $\hat{c}$ by applying spectral clustering on current $\hat{A}$.
\end{enumerate}

In practice, we obtain reasonable results with only a few updates of $\hat{c}$.
The EM updates in steps (1)--(3) also converge quickly given a good estimate of the community label.
For all simulations in Section~\ref{sec: numerical results}, we set $T=2$ and the number of EM iterations to be 20.

\begin{remark}
We note that the alternation of EM updates with community label updates
in the algorithm above leaves something to be desired, in that
it would be preferable to have a single EM algorithm that jointly
optimizes $\hat A, \hat w$ and $\hat c$.
Of course, efforts toward such an algorithm immediately run up against the
well-known fact that the natural EM update in the SBM is computationally
intractable.
In light of this, one might consider
adapting the pseudo-likelihood method proposed in \cite{Amini.et.al.2013},
but this approach only solves the problem of updating
$\hat c$ and $\hat w$ based on (an estimate of) $A$.
Incorporating the observed networks $A^{(1)},A^{(2)},\dots,A^{(N)}$
is non-trivial in light of the fact that these observed networks
are dependent through $A$.
The development of a more principled update procedure,
using pseudo-likelihood or other similarly-motivated approximations
such as variational methods or profile-likelihood,
is a promising avenue for future research,
but one which we do not pursue further here.
\end{remark}

\begin{remark}
Our algorithm is initialized by the majority vote instead of the method of moments that is often used in crowdsourcing \cite{Zhang&Chen&Zhou&Jordan2014}. While the two initializations may have different accuracy, we have found through simulations that there is not much difference once EM has been applied, and most of the time EM improves substantially over the initial value, whichever method is used to initialize. We believe this is because by assuming the block structure, we are able to leverage the information shared among many entries within each block. For simplicity, we only use the majority vote initialization in this paper.
\end{remark}

\subsection{A theoretical guarantee of convergence}
We focus our theoretical investigation of convergence properties on the case $T=1$.
Before stating the result, we need to introduce further notation.
Recall that $\hat{c}$ is the estimate of the label assignment $c$ output by spectral clustering.
Following \cite{Joseph&Yu2013}, we measure the error between $\hat{c}$ and $c$ by
\begin{equation}\label{eq: community detection error}
  \gamma(c,\hat{c}) = \min_{\tilde{c}}\max_{1\le k \le K} \frac{|\{i:\hat{c}_i=k,\tilde{c}_i\neq k\}|+|\{i:\tilde{c}_i=k,\hat{c}\neq k\}|}{|\{i: \tilde{c}_i = k\}|},
\end{equation}
where the minimum is over all $\tilde{c}$ obtained from $c$ by permuting labels of $c$. 

For $x,y\in(0,1/2)$, define
\begin{eqnarray}\label{eq: functions}
\label{eq: function h}  h(x) &:=& \frac{\log(2-2x)}{\log(2-2x) -\log (2x)}, \\
\label{eq: function phi}  \phi(x) &:=& x - h^{-1}\left(\frac{x}{2}+\frac{h(x)}{2}\right), \\
\label{eq: function varphi}  \quad \varphi(x,y) &:=& \Big( x-h(x)\Big)^2\Big( y-h(y)\Big)^2,
\end{eqnarray}
where $h^{-1}$ is the inverse of $h$ and the graph of $h^{-1}$ is shown in Figure~\ref{fig:neighborhood}.
A simple analysis shows that $h$ is an increasing function, $h(x)\ge x$ for every $x\in(0,1/2)$, $\lim_{x\rightarrow 0}h(x)=0$ and $\lim_{x\rightarrow 1/2} h(x)=1/2$. This implies $\phi(x)\ge 0$ for every $x\in(0,1/2)$ and $\lim_{x\rightarrow 0}\phi(x)=\lim_{x\rightarrow 1/2}\phi(x)=0$; similarly, $\lim_{x\rightarrow 0}\varphi(x)=\lim_{x\rightarrow 1/2}\phi(x)=0$. 
For every $\delta\in (0,1/4)$, denote
\begin{equation}\label{eq: R delta}
  R_\delta := \left\{\theta = (x,y,z)^\tran: \delta \le y,z \le 1/2 - \delta \ \text{and} \ \delta \le x  \le 1-\delta \right\}.
\end{equation}

We can now formalize a convergence result for 
the algorithm in Section~\ref{subsec: em-type algorithm}.  
The following theorem establishes exponential convergence of the parameter
estimates for $P,Q$ and $W$.
A proof can be found in Appendix~\ref{ap: convergence}.

%Consider a block $J = \{(i,j): c_i=k,c_j=l\}$ of size $n_k\times n_l$ for some $1\le k,l\le K$. Denote by $w$, $p$ and $q$ the constant entries of restrictions of $W$, $P$ and $Q$ on $J$. Further, let $\theta=(w,p,q)^\tran$ and $\theta^t$ be the estimate of $\theta$ after repeating steps (1)-(3) of the algorithm in Section~\ref{subsec: em-type algorithm} $t$ times.   
\begin{theorem}[Convergence of the EM algorithm]\label{thm: guarantee of EM algorithm}
Consider the algorithm in Section~\ref{subsec: em-type algorithm} with $T=1$. Fix $k,l\in \{1,...,K\}$ and consider the $(k,l)$ block $J = \{(i,j): c_i=k,c_j=l\}$ of size $n_k\times n_l$. Denote by $w$, $p$ and $q$ the common values of entries of $W$, $P$ and $Q$ on $J$, respectively. Further, let $\theta=(w,p,q)^\tran$ and $\theta^t$ be the estimate of $\theta$ after repeating steps (1) through (3) a total of $t$ times.
Assume that $\theta \in R_\delta$ and
\begin{eqnarray*}
  N &\ge& \frac{C}{\delta^2}\max\left\{\log\frac{1}{\delta\phi(p)}, \ \log\frac{1}{\delta\phi(q)},\frac{\log(1/\d)}{\varphi(p,q)}\right\}, \\
  n_k n_l &\ge& \frac{C r^2 N}{\delta^3}\max\left\{\frac{1}{\delta^2},\frac{1}{\phi^2(p)}, \frac{1}{\phi^2(q)}\right\}, \\
  \gamma^2(\hat{c},c) &\le& \frac{\delta}{C}  \max\left\{\delta, \phi(p), \phi(q) \right\},
\end{eqnarray*}
where $C$ is a sufficiently large constant. Then with probability at least $1-\exp(-r)$, 
\begin{eqnarray*}
  \|\theta^t-\theta\| &\le& \exp\left(- t\left[N\delta \varphi(p,q)-\log\frac{N}{\delta^4}\right]\right)\cdot\|\theta^0-\theta\|
  + \frac{\gamma^2(\hat{c},c)}{\delta} + \frac{r\Phi}{\delta},
\end{eqnarray*}
where 
$$
\Phi := \min\left\{\log^2\left(\frac{1}{\delta}-1\right) \sqrt{\frac{N}{n_kn_l}}, \ \ \frac{1}{\d}\cdot\big(1+\d\big)^{-N\varphi(p,q)}+\frac{1}{\sqrt{\d n_kn_l}}\right\}.
$$
\end{theorem}

The error bound in Theorem~\ref{thm: guarantee of EM algorithm} depends critically on $\delta$ (note that $\phi(x)$ tends to zero as $x\rightarrow 0$); a similar dependence of the error bound on $\delta$ appears in the crowdsourcing literature (see, e.g. \cite[ Theorems 1 and 2]{Zhang&Chen&Zhou&Jordan2014}).
As $\delta$ becomes smaller, i.e., $(w,p,q)^\tran$ gets closer to the boundary of the set $(0,1)\times(0,1/2)^2$, the problem of estimating parameters becomes harder, and therefore a larger sample size $N$ is required. In the ultra-sparse regime when the average degree does not grow with $n$, $\delta = O(1/n)$ and $N$ must grow linearly in $n$ in order to maintain a meaningful error bound.
 Although we do not focus on optimizing the dependence on $\delta$, the simulations in Section~\ref{sec: numerical results} suggest that the bound is  not tight in terms of $\delta$; empirically, the algorithm still performs reasonably well when networks are relatively sparse. 

The error bound in Theorem~\ref{thm: guarantee of EM algorithm} consists of three terms. The first term goes to zero exponentially fast as $t\rightarrow\infty$ when $N$ is sufficiently large.
The second term depends on the error $\gamma(\hat{c},c)$ in estimating communities, which is essentially proportional to the inverse of the expected node degree of $A$ when the community signal is sufficiently strong. This can be easily shown using existing results on community detection (see, e.g., \cite{Le&Levina&Vershynin2017}), and we do not develop this further in this paper.  
The last term is a statistical error of order $O(K\min\{\sqrt{N}/n,1/n+\exp(-N)\})$ if all communities are of similar sizes.

Theorem~\ref{thm: guarantee of EM algorithm} implies consistency of the
estimates of $W,P$ and $Q$, as well as vanishing fraction of incorrectly estimated edges of the latent adjacency matrix $A$,
provided the number of iterations $t$, the number of networks $N$
and the block sizes all grow suitably quickly.
\begin{corollary}
For fixed $W,P,Q \in \R^{K \times K}$ and
under suitable growth assumptions on the number of iterations $t$,
the number of networks $N$ and the block sizes $\{n_k\}_{k=1}^K$,
the algorithm of Section~\ref{subsec: em-type algorithm} yields consistent
estimates of $W,P$ and $Q$.
Under slightly stronger growth conditions (essentially that
the number of networks $N$ cannot grow too quickly), the 
estimates $\{ \hat{ \tau }_r \}$ converge to the true $\{ \tau_r \}$
as defined in Section~\ref{sec:known parameters},
i.e., the estimate furnished by the algorithm in
Section~\ref{subsec: em-type algorithm} converges to the
likelihood ratio test derived in Section~\ref{sec: MLE optimal}.
\end{corollary}

\begin{remark}
While Theorem~\ref{thm: guarantee of EM algorithm} is stated
for binary $A$ and $\{ A^{(m)} \}_{m=1}^N$,
a generalization to weighted graphs and a broader class of
edge error models is possible.
Analogues to Theorem~\ref{thm: guarantee of EM algorithm}  can be obtained provided that spectral clustering
of $A$ recovers most community labels correctly under the assumed model,
the observed matrices $\{ A^{(m)} \}_{m=1}^N$ concentrate about $A$,
and the edge distributions for $A_{i,j}$ and $\{ A^{(m)}_{i,j} : m=1,2,\dots,N; i,j =1,2,\dots, n \}$ are well-behaved.   More details are given in Appendix~\ref{ap: weighted}.
\end{remark}

%\begin{description}
%  \item[(1)] M step: compute $$\hat{p}_{ab}=\frac{\sum_{r=0}^{m} r(1-\hat{\tau}_{ab}(r))|\hat{I}_{ab}(r)|}{m\sum_{r=0}^m(1-\hat{\tau}_{ab}(r))|\hat{I}_{ab}(r)|}, \
%  \hat{q}_{ab}=\frac{\sum_{r=0}^m \hat{\tau}_{ab}(r)(m-r)|\hat{I}_{ab}(r)|}{m\sum_{r=0}^m\hat{\tau}_{ab}(r)|\hat{I}_{ab}(r)|}, \
%  \hat{M}_{ab}=\frac{1}{\hat{n}_{ab}}\sum_{r=0}^m \hat{\tau}_{ab}(r)|\hat{I}_{ab}(r)|,$$ where $\hat{n}_{ab}$ is the number of entries of $B(X,\hat{c},a,b)$.
%  \item[(2)] E step: for $0\leq r \leq m$, update  $$\hat{\tau}_r=\frac{\hat{M}_{ab}(1-\hat{q}_{ab})^r\hat{q}_{ab}^{(m-r)}}
%  {\hat{M}_{ab}(1-\hat{q}_{ab})^r\hat{q}_{ab}^{(m-r)} + (1-\hat{M}_{ab})\hat{p}_{ab}^r (1-\hat{p}_{ab})^{(m-r)}}.$$
%\end{description}
%The block estimates of  $M$, $p$, and $q$ are \ $B(\hat{M}_{ab}\mathbf{1}\mathbf{1}^T,\hat{c},a,b)$, $B(\hat{p}_{ab}\mathbf{1}\mathbf{1}^T,\hat{c},a,b)$, and $B(\hat{q}_{ab}\mathbf{1}\mathbf{1}^T,\hat{c},a,b)$. The estimator of $B(A,\hat{c},a,b)$ is the matrix whose entry values are $\mathbf{1}_{\{\hat{\tau}_{ab}(r)>0.5\}}$ for all entries $(i,j)\in\hat{I}_{ab}(r)$, and $r=0,...,m$.
%
%Denote the estimators of $A$, $M$, $p$, and $q$ described above by $\hat{A}^{em}$, $\hat{M}^{em}$, $\hat{p}^{em}$, and $\hat{q}^{em}$ ($em$ is short for EM algorithm). 

\section{Numerical results}\label{sec: numerical results}
In this section we empirically compare performance of several estimators of $A$:
the ``naive'' majority-vote estimate $\hat{A}$ described in
Section~\ref{sec: initial block estimation} (MV), which estimates each
entry of $A$ separately;
the EM estimate we proposed in Section~\ref{subsec: em-type
  algorithm} (EM);
and, in simulations, the oracle estimate described in Section~\ref{sec: MLE optimal}
which uses known parameter values of $W$, $P$, and $Q$
(OP).
To control the false positive rate, we also consider variants EM[T] and OP[T] of EM and OP.
Assuming that parameters are known, OP[T] estimates $A$ by performing the likelihood ratio test on each entry of $A$, as discussed in Section~\ref{sec: likelihood ratio test}.
EM[T] first estimates parameters by EM and then plugs them in as true parameters to perform the likelihood ratio test.
We set the false discovery rate to be $0.05$ for both EM[T] and OP[T].

As discussed in the Introduction, one can obtain an estimate of $A$ by thresholding (at $1/2$) the entries of the low-rank estimate of $\E A^{(1)}$ proposed by \cite{Tang&Ketcha&Vogelstein&Priebe&Sussman2016}.
However, this does not yield a good estimate of $A$ and in fact produces very large errors, on a different scale from all other methods. As a result, we omit it from comparisons in order to be able to plot all the other errors together at an appropriate scale.

For our main algorithm (described in Section~\ref{subsec: em-type algorithm}), we set the number of outer loops to $T = 2$ and the number of EM iterations to 20. This means the algorithm first estimates the community structure using $\hat{A}$ as the input. Once node labels are computed, it estimates all parameters of the model, including the posterior $\hat{\tau}$, by running 20 iterations of the inner loop. The posterior $\hat{\tau}$ is then thresholded to obtain an estimate of the original network $A$. This estimate is used in the second run of the outer loop to update the node labels and subsequently re-estimate all parameters and $A$.   

We first test the methods on synthetic networks and then apply
them to brain fMRI data, the motivating example discussed in the
Introduction.
To initialize EM,  we use regularized spectral clustering  \cite{Amini.et.al.2013,Joseph&Yu2013,Le&Levina&Vershynin2015} to estimate the community labels.
Let $\hat{A}_{\text{reg}} = \hat{A}+0.5 n^{-1} \onevector\onevector^\tran$, $D = \text{diag}(\hat{A}_{\text{reg}}\onevector)$ and $L=D^{-1/2}\hat{A}_{\text{reg}} D^{-1/2}$.
We first compute the $K$ eigenvectors of $L$ that correspond to its $K$ largest eigenvalues. We then apply the $K$-means algorithm on row vectors of the $n\times K$ matrix obtained by stacking the $K$ eigenvectors together to find the community labels.
%The number of iterations for K-means algorithm is set to $20$. 
The $K$-means algorithm is implemented via the MATLAB function \texttt{kmeans} and is run with 20 iterations.  

The performance of all estimators is measured by the false discovery rate (FDR) and the true positive rate  (TPR).
For an estimate $\hat A$ of $A$, FDR and TPR are defined as 
\begin{eqnarray*}
  \operatorname{FDR} = \frac{|\{(i,j): \hat A_{ij}=1, A_{ij}=0\}|}{|\{(i,j): \hat A_{ij}=1\}|}, \qquad
  \operatorname{TPR} = \frac{|\{(i,j): \hat A_{ij}=1, A_{ij}=1\}|}{|\{(i,j): A_{ij}=1\}|}.
\end{eqnarray*}
For each method, we also report the overlap $1-\gamma(\hat{c},c)$ between community assignments $\hat{c}$ and $c$, where $\gamma(\hat{c},c)$ is defined by \eqref{eq: community detection error} and $\hat{c}$ is computed by applying regularized spectral clustering on the estimate of $A$ produced by that method. 
Note that $\gamma(\hat{c},c)$ can be greater than one; in that case we set the overlap to zero. Finally, we report the errors in estimating the false positive, false negative and edge probabilities of EM and MV (the corresponding errors of EM[T] are very similar to that of EM and therefore omitted). For EM, we measure the errors by directly computing the ratios of Frobenius norms:
$$
\frac{\|\hat{W}-W\|_F}{\|W\|_F}, \quad
\frac{\|\hat{P}-P\|_F}{\|P\|_F}, \quad
\frac{\|\hat{Q}-Q\|_F}{\|Q\|_F}.
$$
For MV, we first estimate edge probabilities in each block specified by $\hat{c}$ by the average number of non-zero entries of $\hat{A}$ in that block and then compute the Frobenius norm errors defined above for $W$. To estimate $P$ and $Q$ from MV, for each pair of nodes $(i,j)$, if $\hat{A}_{ij}=0$ then we estimate $P_{ij}$ by $\hat{P}_{ij} = S_{ij}/N$; if $\hat{A}_{ij}=1$, we estimate $Q_{ij}$ by $\hat{Q}_{ij} = 1-S_{ij}/N$. We measure the errors of estimating $P$ and $Q$ by the Frobenius norm ratios computed separately over the zero and non-zero entries of $\hat A$:
$$
\left(\frac{\sum_{(i,j):\hat{A}_{ij}=0} (P_{ij}-\hat{P}_{ij})^2}{\sum_{(i,j):\hat{A}_{ij}=0} P_{ij}^2}\right)^{1/2}, \qquad
\left(\frac{\sum_{(i,j):\hat{A}_{ij}=1} (Q_{ij}-\hat{Q}_{ij})^2}{\sum_{(i,j):\hat{A}_{ij}=1} Q_{ij}^2}\right)^{1/2}.
$$

\subsection{Synthetic data}\label{subsec:synthetic data}
We first test the performance of the estimates on a simple
example of a sample of networks with shared community structure.
We generate the adjacency matrix $A$ from an SBM with $n=300$ nodes and  $K=3$ communities of 100 nodes each.
%Following \cite{Amini.et.al.2013},
We parameterize the $3\times 3$ matrix $B$ of within and between communities edge probabilities of this SBM as
$$  B = \rho_w
\left(
  \begin{array}{ccc}
    1 & \b_w & \b_w \\
    \b_w & 1 & \b_w \\
    \b_w & \b_w & 1 \\
  \end{array}
\right).
$$

%To control the average degree $\l_n = \rho_n(1+\beta)$, we scale $B$ by changing $\rho_n$ appropriately.  \liza{I wrote $\rho_n$ to match $\l_n$, but what is the significance of dependence on $n$ in simulations?  Unless either $\rho_n$ or $\l_n$ are specified in terms of $n$, like $C \log n$, there is probably no need to include $n$ as a subscript.}
The parameter $\rho_w$ controls the overall expected node degree of the model while
$\beta_w$ specifies the ratio of the between-community edge probability to the within-community edge probability.
Smaller values of $\beta_w$ correspond to easier community detection.
Conversely, a larger value of $\rho_w$ indicates more observed edges and therefore an easier community detection problem.
Note, however, that the difficulty of the community detection problem does not directly translate into the difficulty of estimating the underlying true $A$, which is also influenced by $P$ and $Q$.

We  similarly parameterize the $3\times 3$ noise matrices $\mathcal{P}$ and $\mathcal{Q}$ of within- and between-communities false positive and false negative probabilities as
$$  \mathcal{P} = \rho_p
\left(
  \begin{array}{ccc}
    1 & \b_p & \b_p \\
    \b_p & 1 & \b_p \\
    \b_p & \b_p & 1 \\
  \end{array}
\right),
\ \ \
\mathcal{Q} = \rho_q
\left(
  \begin{array}{ccc}
    1 & \b_q & \b_q \\
    \b_q & 1 & \b_q \\
    \b_q & \b_q & 1 \\
  \end{array}
\right).
$$
Thus, $P_{ij} = \mathcal{P}_{c_ic_j}$ and $Q_{ij} = \mathcal{Q}_{c_ic_j}$ for $1\le i,j\le n$.
The overall numbers of false positive and false negative edges are controlled by
parameters $\rho_p$ and $\rho_q$, respectively.  The relative prevalence of false positives and false negatives between communities compared to within communities is controlled by parameters $\beta_p$ and $\beta_q$, respectively.
%\liza{Don't see the point of this, and it does not seem right - the first one should be $s_p(1+\beta_q)$?}
Thus, if $A\equiv 0$ (a network with no edges) then the average degree of a noisy realization of $A$ is $\rho_p(1+2\b_p)(n-1)$;
if $A = \mathbf{1}\mathbf{1}^\tran -\diag(\mathbf{1})$ (a fully connected network), then the average degree of a noisy realization of $A$ is $(n-1)-\rho_q(1+2\b_q)(n-1)$.

In order to focus attention on the relative performance of various methods dealing with a noisy sample of networks, we will use the true number of communities in simulations, $K = 3$.  When the number of communities is not  known,
 it can be estimated from $\hat{A}$ in the first stage by several methods
\cite{Chen&Lei2014,Wang&Bickel2015,Le&Levina2015}, which have been shown to provide accurate results when $K$ is relatively small compared to $n$.   Alternatively, one could use a larger $K$ and interpret the stochastic block model fit as a histogram approximation to the network rather than the true model, as was argued in \cite{Olhede&Wolfe2014}.%\liza{cite the corresponding paper by Olhede and Wolfe}.

The performance of all methods --- majority vote (MV), our proposal (EM, EM[T]) and oracle parameters (OP, OP[T]) --- is shown in Figures~\ref{fig: beta w}, \ref{fig: beta p} and \ref{fig: beta q}. 
In all cases, $n = 300$, $K = 3$, community sizes are equal, $\rho_w=0.15$, $\rho_q=0.2$, $\rho_p=0.25$, the target FDR is set to $0.05$, and all results are averaged over 100 replications.   
To see the effect of structured versus unstructured noise,
we consider three different settings where we fix two of the parameters $\b_w$, $\b_p$, $\b_q$ and let the third one vary.
In Figure~\ref{fig: beta w} the out-in ratios $\b_p=\b_q=1$, meaning that $P$ and $Q$ do not have any community structure and all entries of $A$ are equally likely to be flipped to the opposite.  
When $\b_w$ is not too close to 0 or 1, community labels and parameters of the SBM are accurately estimated, EM performs similarly to the oracle and has a much smaller FDR (essentially equal to the target of 0.05)  than MV.
In contrast, when $\b_w$ is close to 0 or 1, EM does not estimate all SBM parameters accurately, but it still provides a reasonable estimate of $A$.
When likelihood ratio tests are used, both EM[T] and OP[T] output estimates with stable FDR close to the target $0.05$,
although the FDR of EM[T] is slightly larger due to the errors from parameter estimation.  
In most cases, MV has large FDR and TPR, which indicates that it estimates $A$ as having many more edges than it really does. Compared to EM and EM[T], MV also has larger errors in estimating false positive and false negative probabilities.  All methods perform fairly similarly in recovering communities, with MV being the least accurate and OP[T] the most accurate.

Figures~\ref{fig: beta p} and~\ref{fig: beta q} show the effect of false positive and false negative edges 
when one of parameters $\b_p$, $\b_q$ is set to 1 and the other varies.
Again, EM and OP perform similarly when $\b_p$, $\b_q$ are not too large and community labels can be accurately estimated.
EM also has much smaller FDR than MV in all settings. Both EM[T] and OP[T]  have stable FDRs, close to the target of 0.05, but at the expense of lower TPR as $\beta_p$ or $\b_q$ increases.

Overall, as one would expect, all methods perform better as the sample size $N$ increases, $\b_w$, $\b_p$ and $\b_q$ decrease, and the community structure becomes stronger.  
EM and OP perform very similarly and provide better FDR than MV in all settings, especially when $N$ is small.
EM[T] and OP[T] also perform well in controlling the FDR.    These empirical results show the importance of leveraging the block structure for estimating the original network $A$.

%\begin{figure}[!ht]
%  \centering
%  \includegraphics[trim=100 10 10 10,clip,width=1.05\textwidth]{mcc_vs_samplesize_pq_const}\\
%  \caption{The average of MCC over $50$ repetitions as a function of the sample size $N$.
%  In all plots: $\b=0.1$, $\b_p=\b_q=1$, $\a_w=\a_p=\a_q=(1,1)^\tran$, $s_p=s_q=50$.}
%  \label{fig:pqconst}
%\end{figure}
%
%\begin{figure}[!ht]
%  \centering
%  \includegraphics[trim=100 10 10 10,clip,width=1.05\textwidth]{mcc_vs_samplesize_p_const}\\
%  \caption{The average of MCC over $50$ repetitions as a function of the sample size $N$.
%  In all plots: $\b=0.1$, $\b_p = 1$, $\b_q=0.2$, $\a_w=\a_p=\a_q=(1,1)^\tran$, $s_p=s_q=50$.}
%  \label{fig:pconst}
%\end{figure}
%
%\begin{figure}[!ht]
%  \centering
%  \includegraphics[trim=100 10 10 10,clip,width=1.05\textwidth]{mcc_vs_samplesize_q_const}\\
%  \caption{The average of MCC over $50$ repetitions as a function of the sample size $N$.
%  In all plots: $\b=0.1$, $\b_p = 0.2$, $\b_q=1$, $\a_w=\a_p=\a_q=(1,1)^\tran$, $s_p=s_q=50$.}
%  \label{fig:qconst}
%\end{figure}

\begin{figure}[!ht]
  \centering
  \includegraphics[trim=20 30 30 10,clip,width=.8\textwidth]{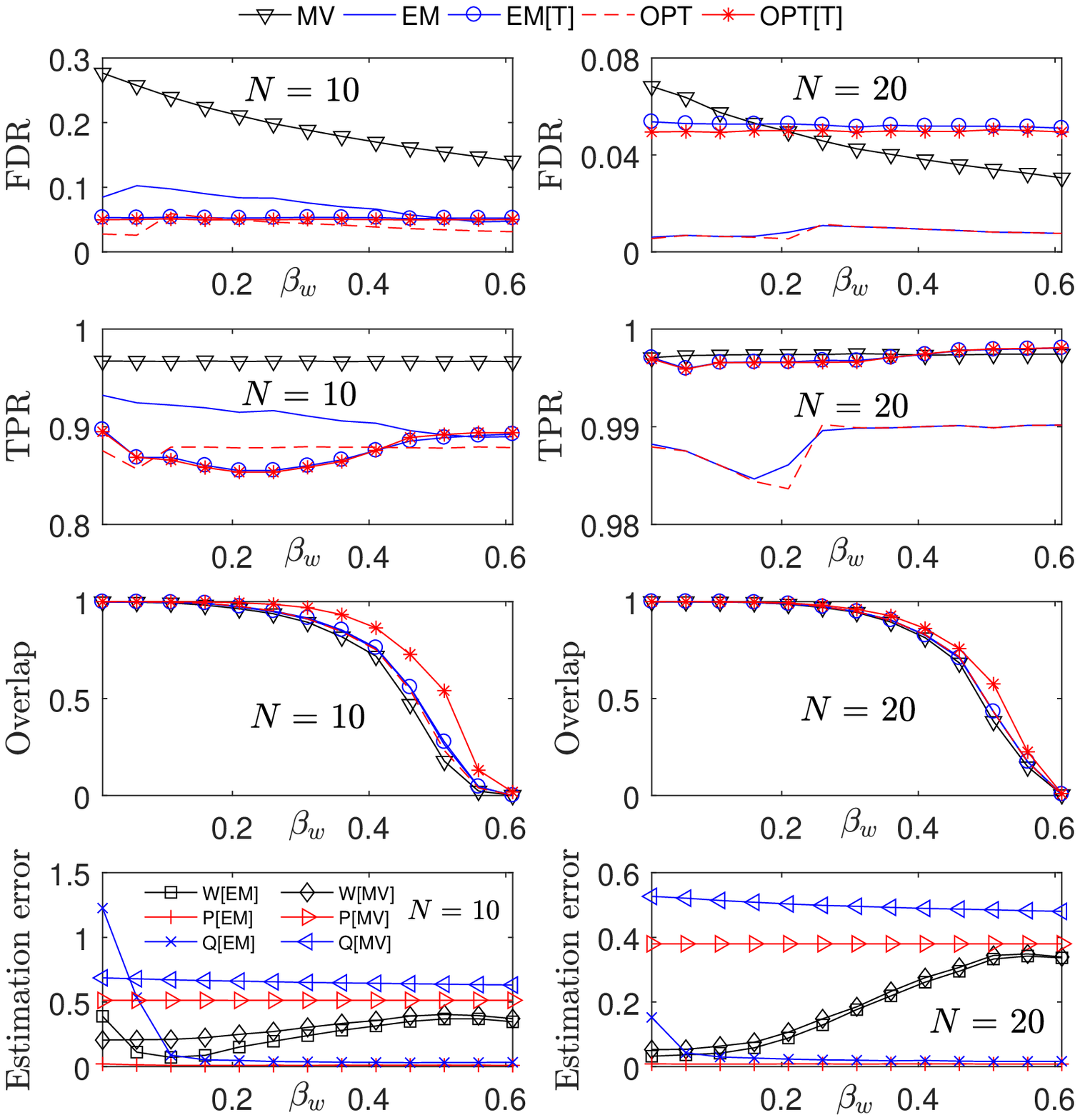}\\
  \caption{A comparison of different methods as the out-in edge probability ratio $\beta_w$ increases and false positives and false negatives occur uniformly at random ($\b_p=\b_q = 1$). First row: false discovery rate for edges of $A$.  Second row: true positive rate for edges of $A$, Third row: overlap  $1-\gamma(\hat{c},c)$ between the true labels $c$ and an estimate $\hat{c}$ obtained from the estimated $A$ for each method.  Fourth row: errors of EM and MV in estimating matrices $W$, $P$ and $Q$.
  %They are averaged over 100 replications when out-in edge probability ratio $\beta_w$ varies from 0 to 0.6 and the number of noisy observations $N$ is either 10 or 20.
For all cases, $\rho_w = 0.15$, $\rho_q=0.2$, $\rho_p=0.25$, $n=300$, $K=3$, and the target false positive rate for both EM[T] and OPT[T] is $0.05$. All measures are averaged over 100 replications.  
}
  \label{fig: beta w}
\end{figure}

\begin{figure}[!ht]%\label{fig: beta p}
  \centering
  \includegraphics[trim=30 30 40 10,clip,width=.8\textwidth]{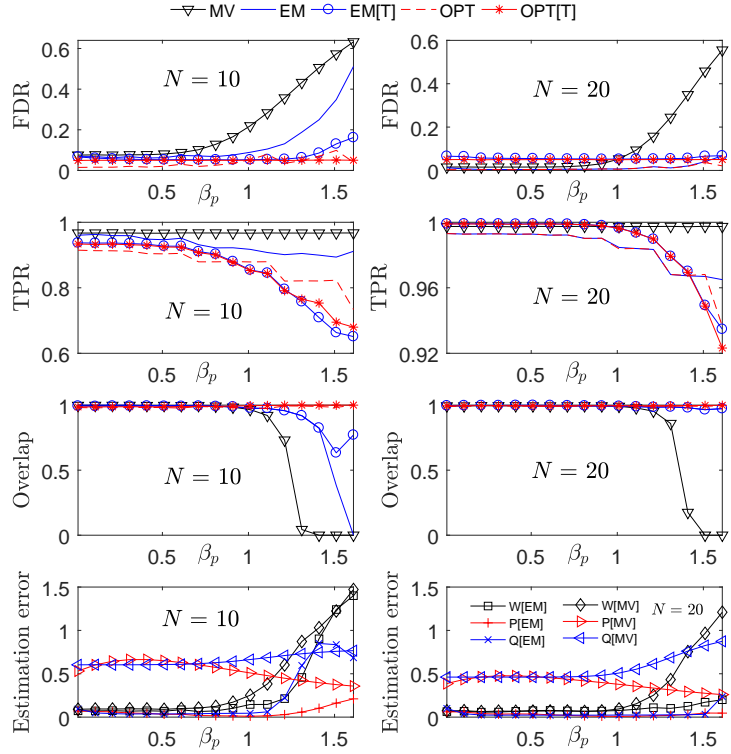}\\
  \caption{
  A comparison of different methods as the out-in false positive probability ratio $\beta_p$ increases, false negatives occur uniformly at random ($\b_q = 1$) and out-in edge probability $\beta_w=0.2$.  First row: false discovery rate for edges of $A$.  Second row: true positive rate for edges of $A$, Third row: overlap  $1-\gamma(\hat{c},c)$ between the true labels $c$ and an estimate $\hat{c}$ obtained from the estimated $A$ for each method.  Fourth row: errors of EM and MV in estimating matrices $W$, $P$ and $Q$.  For all cases, $\rho_w = 0.15$, $\rho_q=0.2$, $\rho_p=0.25$, $n=300$, $K=3$, and the target false positive rate for both EM[T] and OPT[T] is $0.05$. All measures are averaged over 100 replications.  
}
  \label{fig: beta p}
\end{figure}

\begin{figure}[!ht]%\label{fig: beta q}
  \centering
  \includegraphics[trim=20 30 50 10,clip,width=.8\textwidth]{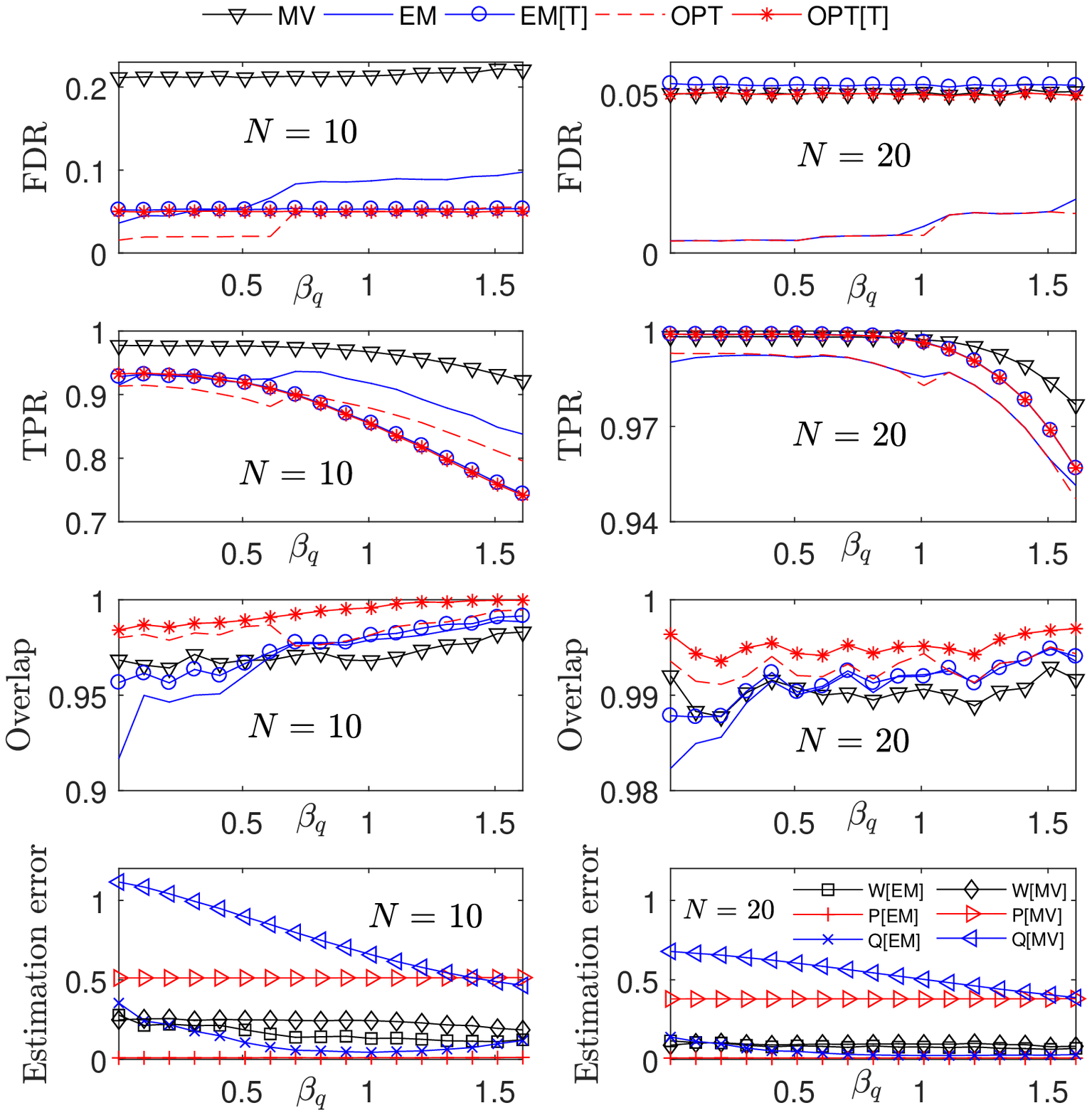}\\
  \caption{A comparison of the performance of different methods as the out-in false negative probability ratio $\beta_q$ increases, false positives occur uniformly at random ($\b_p = 1$) and out-in edge probability is fixed ($\beta_w=0.2$).   First row: false discovery rate for edges of $A$.  Second row: true positive rate for edges of $A$, Third row: overlap  $1-\gamma(\hat{c},c)$ between the true labels $c$ and an estimate $\hat{c}$ obtained from the estimated $A$ for each method.  Fourth row: errors of EM and MV in estimating matrices $W$, $P$ and $Q$. }
  \label{fig: beta q}
\end{figure}

%\begin{figure}[!ht]
%  \centering
%  \includegraphics[trim=100 10 10 10,clip,width=1.05\textwidth]{mplots05}\\
%  \caption{MCC as a function of sample size $m$, where $n_1=n_2=100$, $\lambda=10$, $w_M=[1,1]$, $\beta_M = 0-0.275$ from top to bottom; $w_p=[1,1]$, $\beta_p=0.5$, $a_p=0.1$; $q=0.3\times\mathbf{1}\mathbf{1}^T$.}
%\end{figure}

%\begin{figure}[!ht]
%  \centering
%  \includegraphics[trim=100 10 10 10,clip,width=1.05\textwidth]{mplots10}\\
%  \caption{MCC as a function of sample size $m$, where $\beta_p=1$.}
%\end{figure}

%\begin{figure}[!ht]
%  \centering
%  \includegraphics[trim=90 0 10 10,clip,width=1.05\textwidth]{LambdaMcc}\\
%  \caption{The average of MCC over 100 replications as a function of $\lambda$, with $n_1=n_2=100$, $w_M=(1,1)$, $\beta_M = 0.4$, $w_p=(1,1)$, $\beta_p=0.5$, $a_p=0.1$, $w_q=(1,1)$, $\beta_p=1$, and $a_q=0.3$.}
%\end{figure}

\subsection{Brain networks}\label{subsec: brain network}
In this section we evaluate the performance of our proposed EM method on functional brain networks
\cite{Taylor&Chen&et.al.2011,Taylor&Demeter&et.al.2014}.
%In the first stage of EM we use a spectral method of \cite{Le&Levina2015} to estimate the number of communities from $\hat{A}$.
The data are obtained from resting state fMRI images, where blood oxygenation levels at different locations in the brain are recorded over time.
These time series of oxygen levels are then preprocessed and used to compute a Pearson correlation between each pair of locations.  Finally, the correlations are thresholded  to construct a binary network matrix.  

The dataset we analyze here includes $81$ subjects, $39$ suffering from schizophrenia and $42$ healthy controls (see \cite{Taylor&Chen&et.al.2011,Taylor&Demeter&et.al.2014} for details on the data). The resulting correlation matrices are $264\times 264$, corresponding to 264 regions of interest (ROIs) in the brain.    For a given value of the threshold $\nu\in(0,1)$, we construct a brain network $A^{(m)}$ for subject $m$ from its correlation matrix $C^{(m)}$
by setting $A^{(m)}_{ij} = 1$ if $|C^{(m)}_{ij}|>\nu$ and $A^{(m)}_{ij} = 0$ otherwise.
We view each network $A^{(m)}$ as a noisy observation of an underlying true biological network $A$, which differs between schizophrenics and controls.   %Our goal is to estimate $A$ from $A^{(m)}$.
Note that the number of edges in $A^{(m)}$ depends on $\nu$.
In practice, there is no consensus on how to choose $\nu$,
therefore it is desirable to have a method that is accurate and stable over a large range of  $\nu$ values.

Since the number of communities $K$ is unknown, we first estimate it from the majority vote matrix $\hat{A}$ using a spectral method for estimating the number of communities based on counting the negative eigenvalues of the graph Bethe-Hessian \cite{Le&Levina2015}.
As expected, the estimated number of communities depends on the threshold value; please see Figure~\ref{fig:brain stability}.     However, there is a stable range of $\nu$ roughly between 0.2 and 0.4, and the estimated $K$ over that range is close to 14, the number of functional regions suggested independently by \cite{Power&Cohen&et.al.2011}.     To facilitate comparison with this known functional parcellation, we fit our EM-based method with $K=14$ in the subsequent analysis.

%\begin{figure}[!ht]
%  \centering
%  \includegraphics[trim=0 0 30 10,clip,width=.6\textwidth]{nc}\\
%  \caption{Number of communities estimated by the Bethe-Hessian based spectral method from majority vote matrix $\hat A$ as a function of the threshold $\tau$.
%}
%\label{fig: nc}
%\end{figure}

Figure~\ref{fig:brain stability}
shows several global summary statistics of the estimates of $A$ for a range of $\nu$.   Global network summary statistics have been a popular tool in the study of brain networks \cite{Rubinov&Sporns2010} and can be used to predict disease status, but here our focus is on how the network estimation method affects the population estimates of these summary statistics.  In particular, since there is no consensus on choosing $\nu$, a stable estimate over a range of values of $\nu$ is desirable.   Overall, the plots in Figure~\ref{fig:brain stability} show  that the EM method produces much more stable estimates over a wider range of $\nu$.     For all statistics, the left column shows schizophrenics and the right column healthy controls.   The first row shows estimated population average degree for the EM and MV methods, along with the range (from minimum to maximum value) and sample median of individual's average network degrees.   Two other summary statistics shown in the second and third rows, global efficiency (the average inverse shortest path length, viewed as a measure of network functional integration) and transitivity (a normalized average fraction of triangles around an individual node, viewed as 
 a measure of network functional segregation that reflects the presence of communities), are also more stable over a wider range of $\nu$ for the EM method.    The fourth row shows the strength of the networks estimated by EM and MV, measured by modularity optimized via spectral clustering  \cite{Newman2004}. 
%Again, the plots show that the community strength is more stable for the EM method. 

Finally, the fifth row presents the estimated number of communities  based on counting the negative eigenvalues of the graph Bethe-Hessian \cite{Le&Levina2015}.     For almost all summaries, the estimates obtained by EM are closer to the median values obtained from the individual networks, suggesting the EM produces a more accurate population estimate, or at least one that is more representative of the sample.   The exception to this general pattern occurs only at very low values of $\nu$, where the network is likely too dense to be informative.

\begin{figure}[!ht]
\centering
\subfiguretopcaptrue
\subfigure{
    \includegraphics[trim=50 0 50 30,clip,width=1\textwidth]{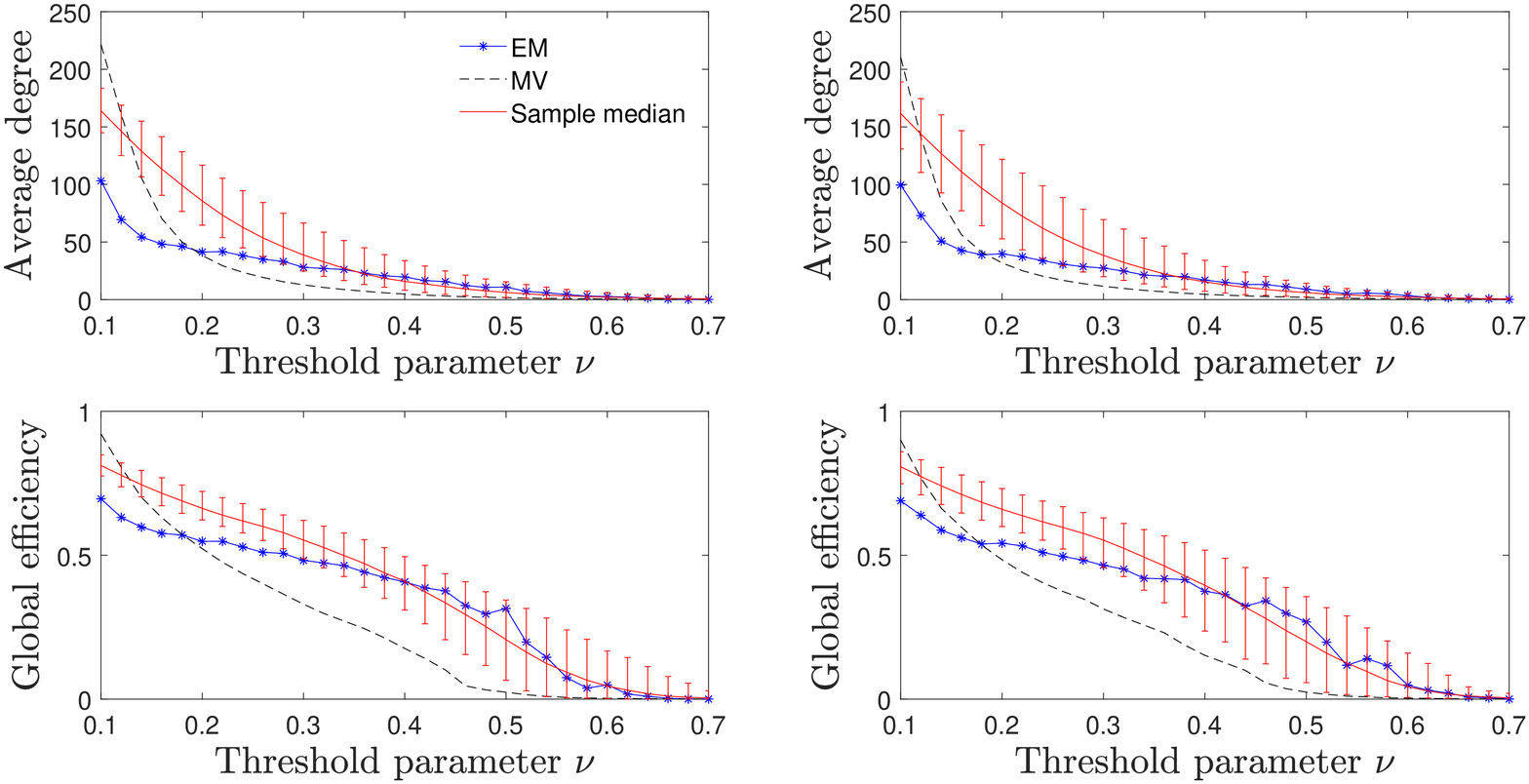}
}\\
\subfigure{
    \includegraphics[trim=50 10 50 30,clip,width=1\textwidth]{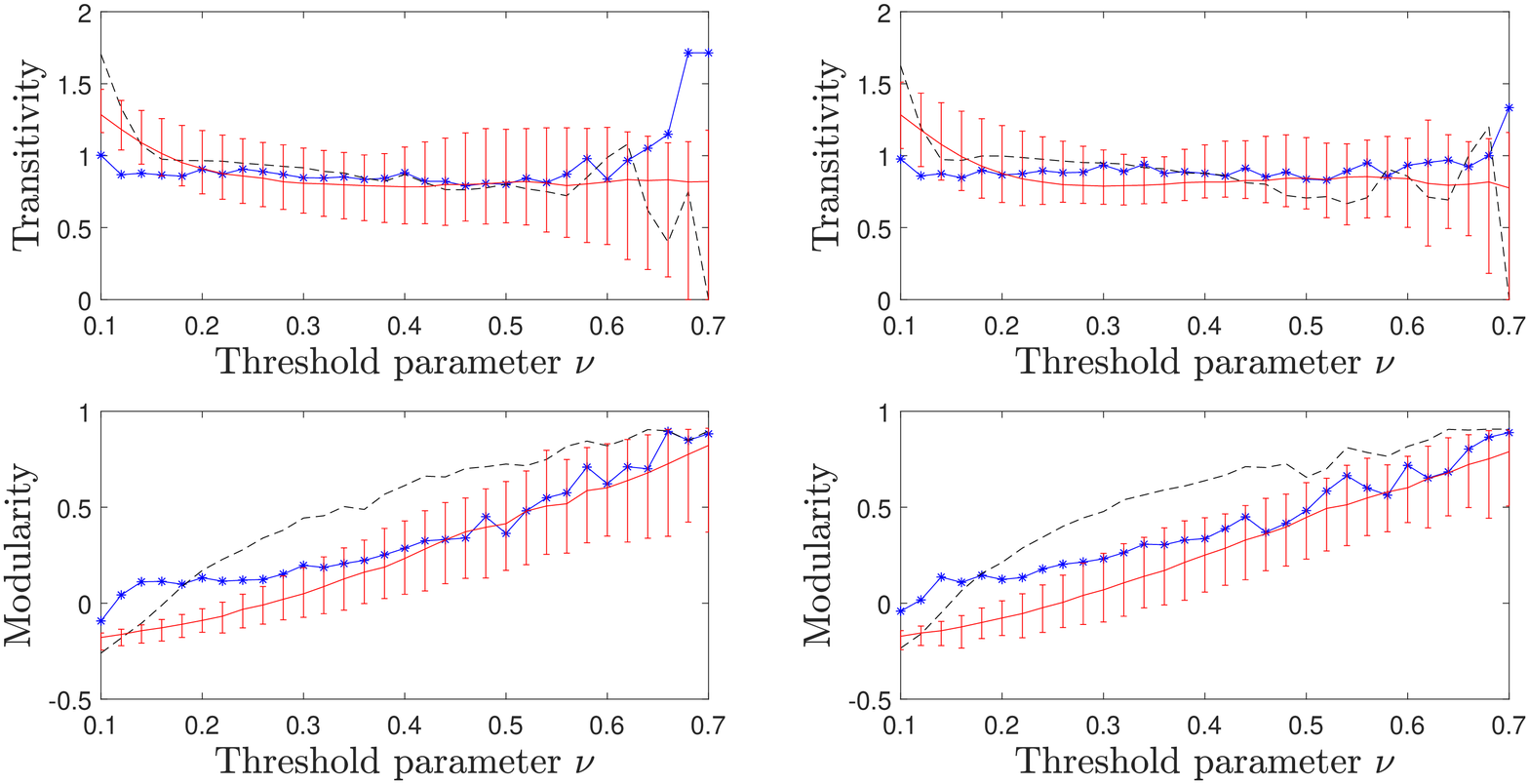}
}\\
\subfigure{
    \includegraphics[trim=0 0 0 25,clip,width=1\textwidth]{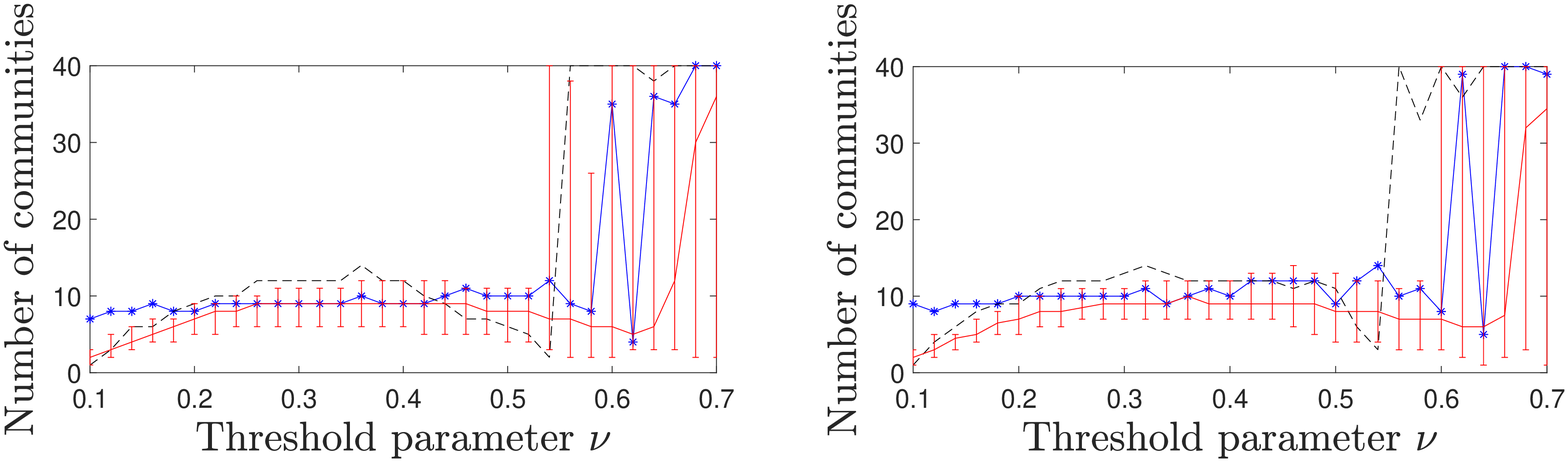}
}
\caption{Global summary statistics of brain networks estimated by EM and MV as the threshold parameter $\nu$ changes.
         Left column: schizophrenics; right column:  healthy controls.   First row:  average degree.  Second row: global efficiency.   Third row:  transitivity.   Fourth row:  modularity value of the community labels estimated by spectral clustering. Fifth row: number of communities estimated by the Bethe-Hessian based spectral method from majority vote matrix $\hat A$. The sample median refers to the median of values computed from individual networks after thresholding, and is shown together with the range.  
}
\label{fig:brain stability}
\end{figure}

%\begin{figure}[!ht]
%\centering
%\includegraphics[trim=0 0 0 10,clip,width=0.6\textwidth]{clustering}
%\caption{The accuracy of estimating the node labels of EM algorithm (EM) and spectral clustering that uses the majority vote matrix $\hat{A}$ as the input (MV). The two methods are applied on subsets of brain data of sick and healthy people. The accuracy is measured  by $\|\hat{M}-M\|_F$, the Frobenius norm of the difference between two co-clustering matrices $\hat{M}$ and $M$ that correspond to the node labels estimate of either method and the ground truth, respectively.}
%\label{fig: clustering}
%\end{figure}

Figure~\ref{fig: brain network illustration} shows sagittal views of the underlying networks estimated by EM and MV
for the threshold parameter $\nu=0.5$.   We use $\nu = 0.5$ since higher $\nu$ produces sparser networks that are easier to visualize, and the network statistics are still fairly stable in that range.   The plots are drawn by the brain network visualization tool \textit{BrainNet Viewer} \cite{Xia&Wang&He2013}.
\begin{figure}[!ht]
\centering
\subfiguretopcaptrue
\subfigure{
    \includegraphics[trim=70 40 50 20,clip,width=0.47\textwidth]{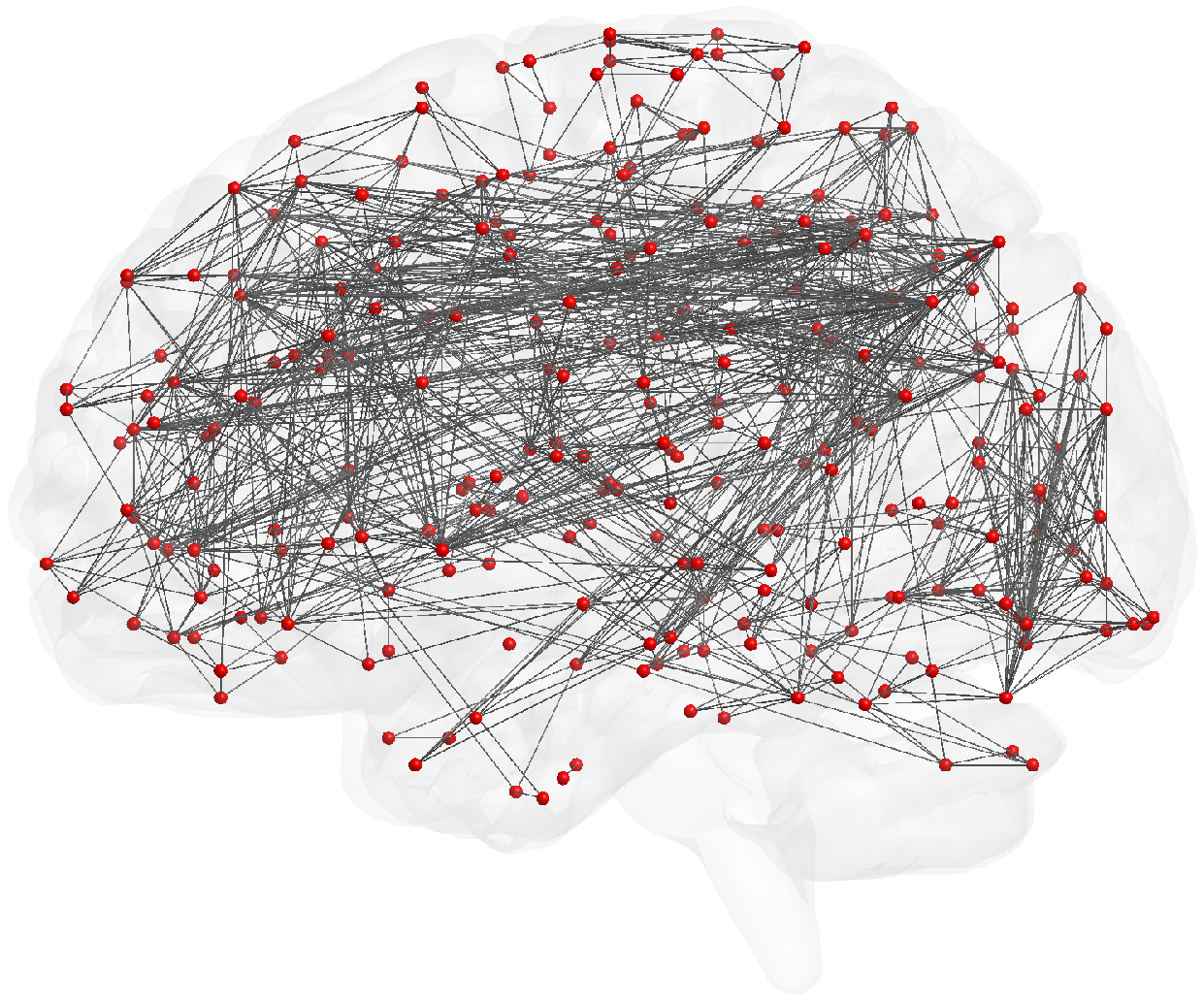}
}
\subfigure{
    \includegraphics[trim=70 40 50 20,clip,width=0.47\textwidth]{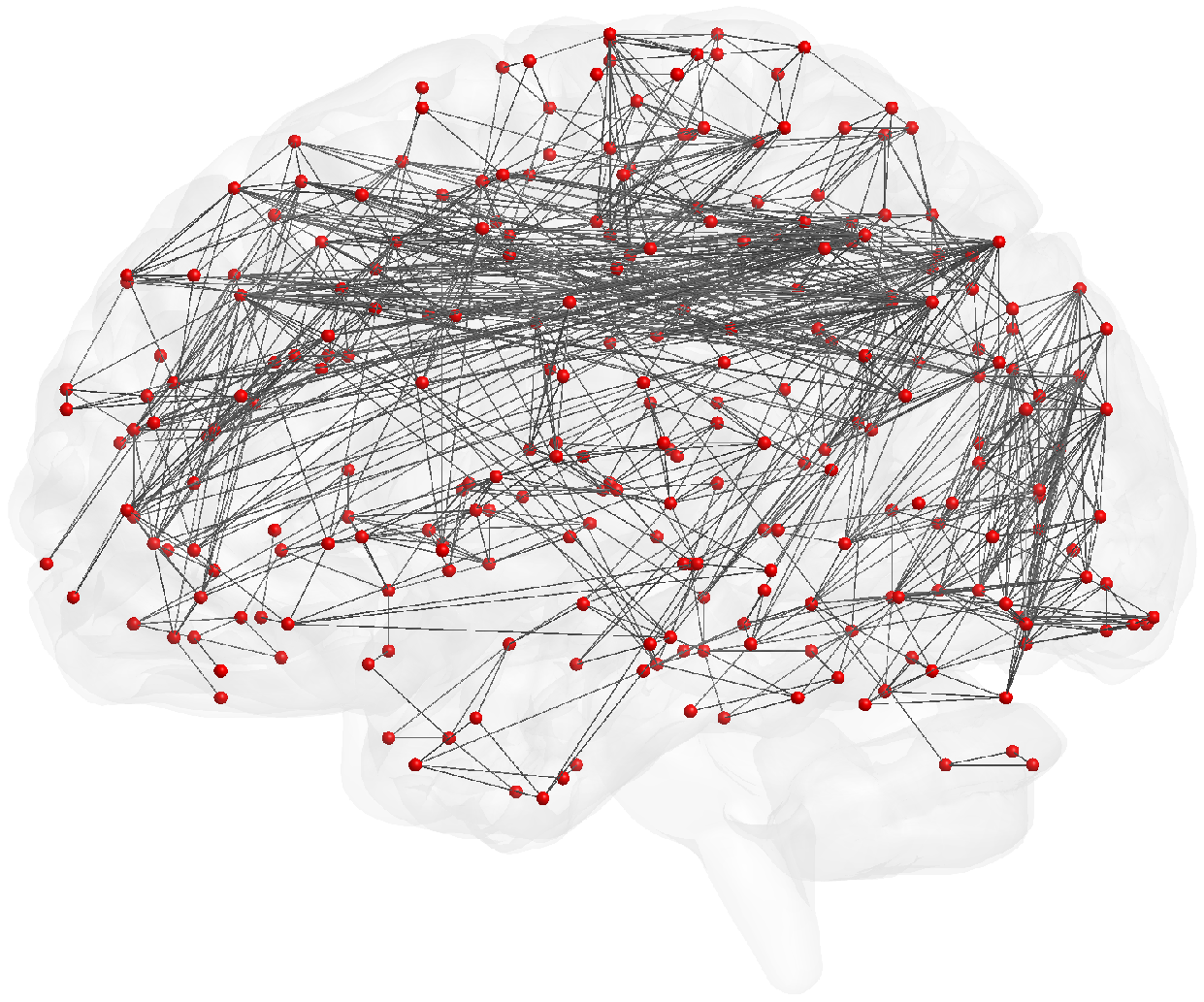}
}\\
\subfigure{
    \includegraphics[trim=70 40 50 20,clip,width=0.47\textwidth]{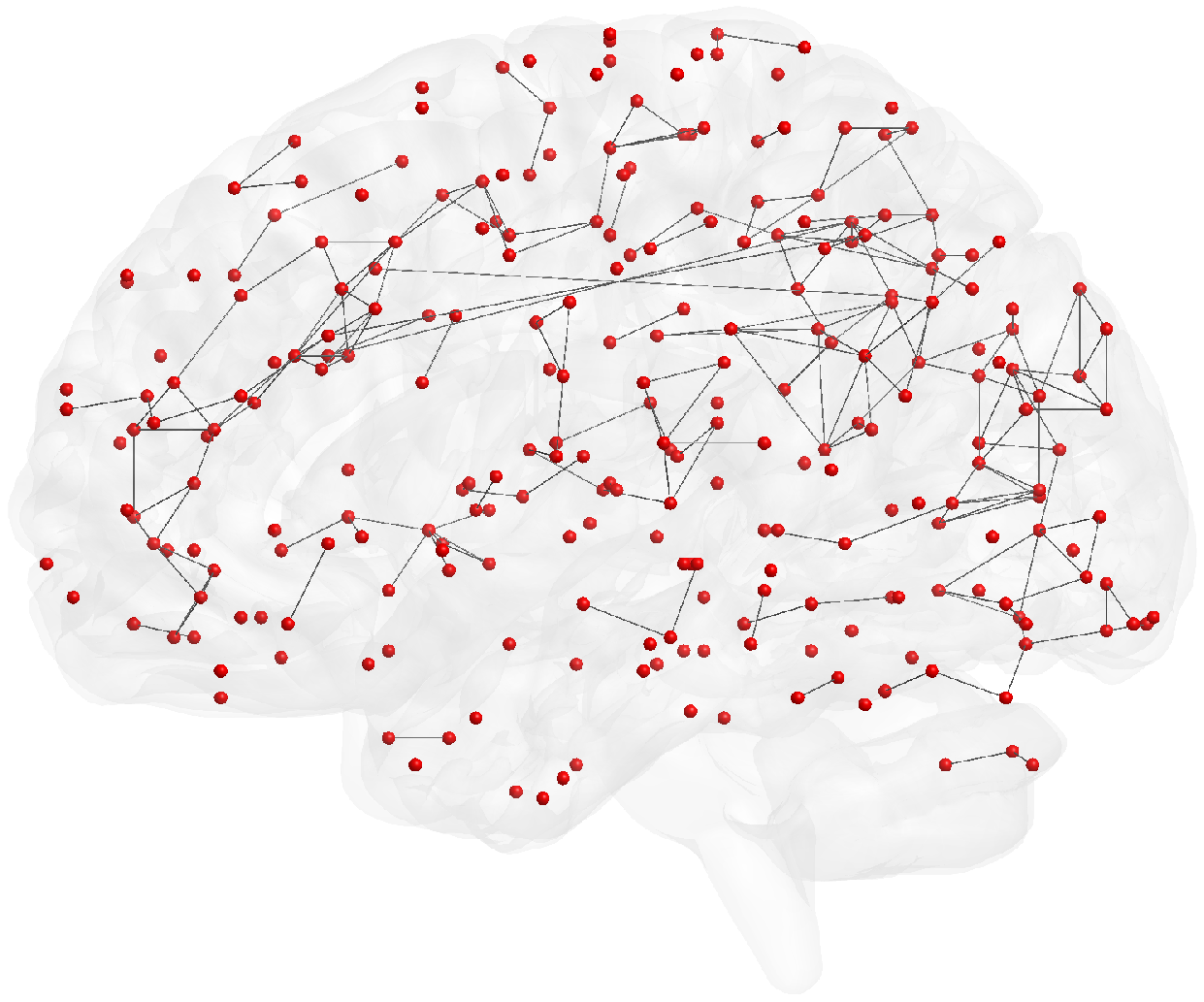}
}
\subfigure{
    \includegraphics[trim=70 40 50 20,clip,width=0.47\textwidth]{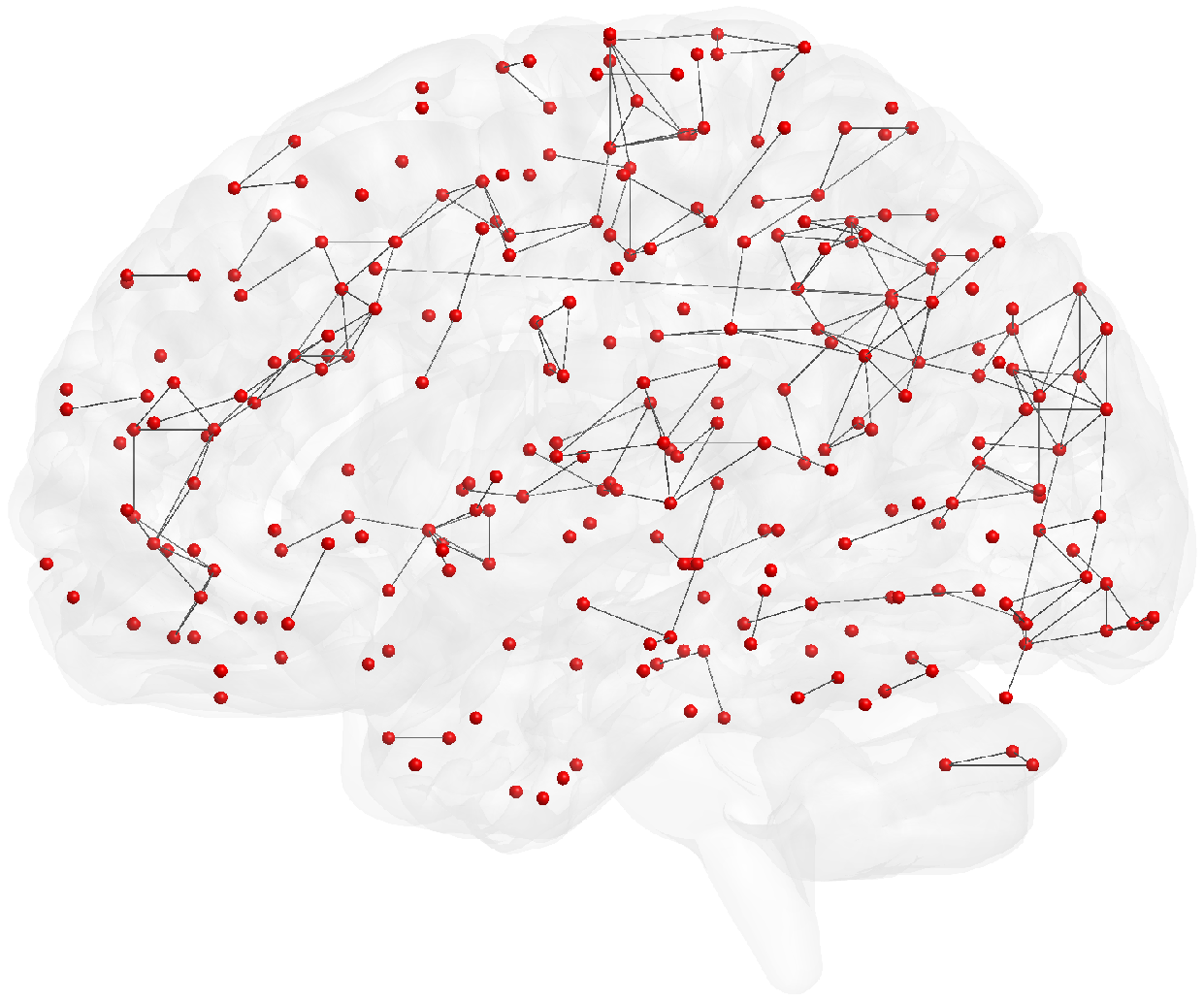}
}
\caption{Estimates of underlying population networks with threshold $\nu=0.5$.  
         Left column: schizophrenics;  right column: healthy controls.   Top row:  EM;  bottom row:  MV.}  
\label{fig: brain network illustration}
\end{figure}

\section{Discussion}\label{sec:discussion}
We have proposed a novel way to estimate an underlying ``population'' network from its multiple noisy realizations, leveraging the underlying community structure.  In contrast to most previous work (with the notable exceptions of \cite{Tang&Ketcha&Vogelstein&Priebe&Sussman2016,Wang&Zhang&Dunson2017}), our algorithm does not vectorize the network or reduce it to global summaries;  the procedure is  designed specifically for network data, and thus tends to outperform methods that do not respect the underlying network structure.   
While we focused on the stochastic block model as the underlying network structure, because of its simple form and its role as an approximation to any exchangeable network model, this assumption is not essential.   An extension to the degree-corrected stochastic block model is left as future work, and we believe in practice the algorithm will work well for any network with community structure.      On the other hand, the assumption of independent noise is important and unlikely to be relaxed.   The assumption of false positive and false negative probability matrices being piecewise constant is also important, as it allows us to significantly reduce the number of parameters and estimate them using the shared information within each block, but clearly many other ways to impose sharing information are possible, perhaps through a general low rank formulation.  We leave exploring such a formulation for future work.

\section*{Acknowledgements}

This work was partially supported by NSF grant DMS-1521551 and ONR grant N000141612910 to E. Levina.   We thank our collaborators in Stephan Taylor's lab in Psychiatry at the University of Michigan for providing a processed version of the data.   

\appendix
\section{Estimation error} \label{ap: estimation error}
We first prove Lemma \ref{lem: estimation error}, which formalizes the intuitive fact that the estimation error is an increasing function of noise levels. %  \liza{To be honest, I am not entirely sure what the point of this lemma is.   It obviously somehow confirms that what we are doing is reasonable, but exactly how this fact supports what we do is not really ever discussed.   Something about it should be added in the main text.   Also, I would move the statements of both lemmas to main text and leave just the proof in the appendix.}  \can{This is discussed a bit in Section 2.2. I have moved these lemmas to that section.}
Recall that for a fixed pair of nodes $(i,j)$, $s = \sum_{m=1}^N A_{ij}^{(m)}$ and $\mu$ is the threshold defined in \eqref{eq: maximum likelihood estimator, known parameters}.

\begin{proof}[Proof of Lemma~\ref{lem: estimation error}]
Denote by $f=f(w,p,q)$ the estimation error \eqref{eq: estimation error}, that is
$$
f(w,p,q) = \P(a^*\neq a) = w\P(s<\mu|a=1) + (1-w)(1-\P(s< \mu|a=0)).
$$
We show that there exists a finite set $\MM\subseteq[0,1/2]$ such that
the partial derivative $\partial f/\partial q$ is positive for all $q\in[0,1/2]\setminus\MM$ and $f$ is continuous for all $q$.
This clearly implies that $f$ is an  increasing function of $q$; the proof that $f$ is increasing in $p$ is similar.

%Recall that $s\sim\text{Binomial}(n,1-q)$ given $a=1$ and $s\sim\text{Binomial}(n,p)$ given $a=0$.
Let $\MM$ be the set of points $q\in[0,1/2]$ such that $\mu=\mu(w,p,q)$ is an integer; this set is finite because $\mu$ is a smooth function of $q$.
Fix $q_0\not\in\MM$ and choose an integer $k$ so that $k<\mu(w,p,q_0)<k+1$. For any $q$ sufficiently close to $q_0$, the event $s<\mu$ is the same as $s\le k$. Since $s\sim\text{Binomial}(N,1-q)$ given $a=1$ and $s\sim\text{Binomial}(n,p)$ given $a=0$, we have
\begin{eqnarray}
\label{eq: binomial cdf 1} \P(s\le k|a=1) &=& (n-k)\binom{N}{k}\int_0^q t^{n-k-1}(1-t)^k dt, \\
\label{eq: binomial cdf 0} \P(s\le k|a=0) &=& (n-k)\binom{N}{k}\int_0^{1-p} t^{n-k-1}(1-t)^k dt.
\end{eqnarray}
It follows that
$$
\frac{\partial f }{\partial q} (w,p,q_0)= w(n-k)\binom{N}{k} q_0^{n-k-1}(1-q_0)^k > 0.
$$

Let us now fix $q_0\in\MM$ and choose an integer $k$ such that $\mu(w,p,q_0) = k$.
We consider four possible cases based on the local behavior of $\mu$ near $q_0$: $\mu$ reaches local maximum at $q_0$, $\mu$ reaches local minimum at $q_0$, $\mu$ is increasing, and $\mu$ is decreasing.

If $\mu$ reaches local maximum at $q_0$ then for any $q$ sufficiently close to $q_0$, the event $s<\mu(w,p,q)$ is the same as $s\le k-1$.
Using \eqref{eq: binomial cdf 1} and \eqref{eq: binomial cdf 0} we obtain that $f$ is continuous at $q_0$.
Similarly, $f$ is continuous at $q_0$ if $\mu$ reaches local minimum at $q_0$.

If $\mu$ is increasing near $q_0$ then for any $q$ sufficiently close to $q_0$, the event $s<\mu(w,p,q)$ is the same as $s\le k-1$ if $q\le q_0$ and $s\le k$ if $q>q_0$. Therefore the jump of $f$ at $q_0$ is
$$
h = w\P(s=k|a=1) - (1-w)\P(s=k|a=0).
$$
Using $s|a=1\sim\text{Binomial}(N,1-q_0)$ and $s|a=0 \sim\text{Binomial}(N,p)$,
a simple calculation shows that $k = \mu(w,p,q_0)$ is equivalent to $h=0$, which implies the continuity of $f$ at $q_0$.
Similarly, $f$ is continuous at $q_0$ if $\mu$ is decreasing near $q_0$, and the proof is complete.
\end{proof}

\begin{proof}[Proof of Lemma~\ref{lem: monotonicity of fdr}]
Recall that $T_\a$ rejects the null hypothesis if $s> k_\a$;
when $s=k_\a$ it rejects the null with some probability $\eta_\a$ adjusted to achieve confidence level $\a$.
Since $s$ follows $\text{Binomial}(N,p)$ if $a=0$ and $\text{Binomial}(N,1-q)$ if $a=1$, the confidence level and power of $T_\a$ satisfy
\begin{eqnarray*}
  \a &=& \eta_\a \binom{N}{k_\a} p^{k_\a} (1-p)^{N-k_\a} + \sum_{m=k_\a+1}^N  \binom{N}{m} p^{m} (1-p)^{N-m},\\
  \gamma_\a &=& \eta_\a \binom{N}{k_\a} (1-q)^{k_\a} q^{N-k_\a} + \sum_{m=k_\a+1}^N  \binom{N}{m} (1-q)^{m} q^{N-m}.
\end{eqnarray*}
Solving for $\eta_\a$ from the first equation and plugging it into the second equation, we obtain
$$
\frac{\gamma_\a}{\a} = \left(\frac{1-q}{p}\right)^{k_\a}\left(\frac{q}{1-p}\right)^{N-k_\a}+
\frac{1}{\a}\sum_{m = k_\a+1}^N\binom{N}{m} (1-q)^m q^{N-m}\left[1-\left(\frac{pq}{(1-p)(1-q)}\right)^{m-k_\a}\right].
$$
Note that $k_\a$ is a piecewise constant function of $\a$.
On every interval of $\a$ where $k_\a$ is constant, the coefficient of $1/\a$ in the above representation of $\gamma_\a/\a$ is positive
because
$$
1-\left(\frac{pq}{(1-p)(1-q)}\right)^{m-k_\a} \ge 0
$$
for every $m\ge k_\a+1$ by the assumption $p,q\le 1/2$.
This implies that $\gamma_\a/\a$ is decreasing on  every such interval,
and in turn on the whole interval $(0,1]$ because $\gamma_\a/\a$ is a continuous function of $\a$.
Since the false positive rate $\xi_\a$ is a decreasing function of $\gamma_\a/\a$ by \eqref{eq: FDR},
the claim of Lemma~\ref{lem: monotonicity of fdr} follows.
\end{proof}

\section{Convergence of the EM algorithm}\label{ap: convergence}
In this section we prove Theorem~\ref{thm: guarantee of EM algorithm}, establishing the convergence of our algorithm described in Section~\ref{subsec: em-type algorithm}.
The proof consists of two steps: showing the convergence of population-level updates (Section~\ref{sec: population-level updates}) and bounding the error between population-level updates and sample-level updates (Section~\ref{sec: sample-based updates}).

%We first prove in Section~\ref{sec: population-level updates} the convergence of population-level updates  borrowing some techniques from \cite{Balakrishnan&Wainwright&Yu}.
%In Section~\ref{sec: sample-based updates} we bound the error of the mean computed from the sample via Talagrand's comparison theorem \cite[Theorem 4.12]{Ledoux&Talagrand1991}.  Finally, we prove a bound on the additional error that results from replacing $c$ with an estimated $\hat{c}$, assuming consistency of community detection.

\subsection{Population-level updates}\label{sec: population update}\label{sec: population-level updates}
\subsubsection{Preliminaries}
We first briefly recall the population-level updates of our algorithm and set up additional notation. To simplify notation, let us fix a pair of nodes $(i,j)$ and denote $a = A_{ij}$, $s=S_{ij}=\sum_{m=1}^N A_{ij}^{(m)}$, $w =W_{ij}$, $p = P_{ij}$, $q=Q_{ij}$ and $\theta = (w,p,q)^\tran$.
Recall that
\begin{equation*}
  w = \P(a=1), \quad p = \P(A_{ij}^{(m)}=1|a=0), \quad q = \P(A_{ij}^{(m)}=0 |a=1),
\end{equation*}
and $s$ follows a mixture of binomial distributions, namely 
\begin{equation}\label{eq: def of s}
s\sim w  \text{Binomial}(N,1-q) + (1-w)  \text{Binomial}(N,p).
\end{equation}

Let $f_{\theta}$ be the joint likelihood of $s$ and $a$. Assume that  $f_{\theta}$ belongs to a parametric family $\{f_{\theta'}| \theta': = (w',p',q')^\tran\in \Theta\}$, with $\Theta$ to be specified. For each $\theta'\in\Theta$, the joint likelihood $f_{\theta'}(s,a)$ of $s$ and $a$ has the form 
$$
f_{\theta'}(s,a) = \big[X_{\theta'}(s)\big]^{a} \big[Y_{\theta'}(s)\big]^{1-a},
$$
where 
\begin{equation}\label{eq: def of X and Y}
X_{\theta'}(s) = w'(1-q')^{s} (q')^{N-s}, \qquad
Y_{\theta'}(s) = (1-w')(p')^{s} (1-p')^{N-s}.
\end{equation}
Summing over $a$, we obtain the marginal likelihood 
$f_{\theta'}(s) = X_{\theta'}(s) + Y_{\theta'}(s)$ of $s$. 
For each $\theta''\in \Theta$, the conditional expectation of $\log f_{\theta''}(s,a)$  given $s$ takes the form
\begin{eqnarray}
 \label{eq: iteration function} T_{\theta''|\theta'}(s) &=& \frac{X_{\theta'}(s)}{X_{\theta'}(s) + Y_{\theta'}(s)}\Big[\log w''+s\log (1-q'') + (N-s)\log q''\Big]  \\
 \nonumber  &+& \frac{Y_{\theta'}(s)}{X_{\theta'}(s) + Y_{\theta'}(s)}\Big[\log(1-w'')+s\log p'' + (N-s)\log(1-p'')\Big].
\end{eqnarray}
The population-level update of a current parameter  estimate $\theta'$ is computed by
\begin{equation}\label{eq: population-level update}
  M(\theta')=\argmax_{\theta''} \E T_{\theta''|\theta'}(s).
\end{equation}

To compute $\hat{M}(\theta')$, let us denote by $F,G$ and $L$ the following functions:
\begin{equation}\label{eq: FGL}
F_{\theta'}(t) = \frac{X_{\theta'}(t)}{X_{\theta'}(t)+Y_{\theta'}(t)}, \
G_{\theta'}(t) = \frac{tX_{\theta'}(t)}{X_{\theta'}(t)+Y_{\theta'}(t)}, \
L_{\theta'}(t) = \frac{tY_{\theta'}(t)}{X_{\theta'}(t)+Y_{\theta'}(t)},
\end{equation}
where $X_{\theta'}$ and $Y_{\theta'}$ are defined in \eqref{eq: def of X and Y}.
Setting partial derivatives of $T_{\theta''|\theta'}$ to zero, we find that the components of $M(\theta')$, which we denote by $M(w')$, $M(p')$ and $M(q')$, respectively, can be computed by
%\begin{equation}\label{eq: population updates}
%M(w') = \E \frac{X_{\theta'}(s)}{X_{\theta'}(s) + Y_{\theta'}(s)}, \ \
%M(p') = \frac{\E \frac{sY_{\theta'}(s)}{X_{\theta'}(s) + Y_{\theta'}(s)}}{\E \frac{NY_{\theta'}(s)}{X_{\theta'}(s) + Y_{\theta'}(s)}}, \ \
%M(q') = \frac{\E \frac{(N-s)X_{\theta'}(s)}{X_{\theta'}(s) + Y_{\theta'}(s)}}{\E \frac{NX_{\theta'}(s)}{X_{\theta'}(s) + Y_{\theta'}(s)}}.
%\end{equation}
\begin{equation}\label{eq: population updates}
M(w') = \E F_{\theta'}(s),
\quad \hat{M}(p') = \frac{\E L_{\theta'}(s)}{ N- N \E F_{\theta'}(s)},
\quad M(q') = \frac{N \E F_{\theta'}(s) - \E G_{\theta'}(s)}{N \E F_{\theta'}(s)}.
\end{equation}
It follows from a simple calculation and \eqref{eq: population updates} that $M(\theta) = \theta$.

%In the finite sample, instead of $\E T(\theta''|\theta')$, we compute the average of $T(\theta''|\theta')$ over all pairs of nodes $(i',j')$ within the block that $(i,j)$ belongs to. 
%The sample update is then the maximizer of this average.

\subsubsection{Guarantee of convergence}

We show that the map $\theta'\mapsto M(\theta')$ is a contraction in a neighborhood of $\theta$.
To specify such a neighborhood, we need additional notation.
For $x,y\in[0,1/2]$, define
\begin{equation}\label{eq: function H}
  H(x,y) := \frac{\log\frac{1-x}{y}}{\log\frac{1-x}{y}+\log\frac{1-y}{x}}.
\end{equation}
Note that the function $h$ defined in \eqref{eq: function h} satisfies $h(x)=H(x,1/2)$. It is easy to check that $H(x,y)$ is increasing in $x$ and decreasing in $y$.
Moreover, $H(x,y)\ge x$ and $H(x,y)\le 1-y$ for all $x,y\in[0,1/2]$.

For $p,q \in [0,1/2]$ and $\e\in[0,1)$, define a neighborhood of $(p,q)$ by
\begin{equation}\label{eq: neighborhood of p* q*}
  U_\e(p,q) = \left( h^{-1}\Big(\e h(p)+(1-\e)p\Big),\frac{1}{2}\right)\times \left(h^{-1}\Big(\e h(q)+(1-\e)q\Big),\frac{1}{2}\right),
\end{equation}
where $h^{-1}$ is the inverse function of $h$ (see Figure~\ref{fig:neighborhood}).
Since $h(x)=H(x,1/2)\ge x$ and $h$ is increasing, it follows that $h^{-1}(x)\le x$ for all $x\in[0,1/2]$.
Therefore $U_\e(p,q)$ is a rectangle containing $(p,q)$ and $U_{\e_1}(p,q)\subseteq U_{\e_2}(p,q)$ if $\e_1\ge \e_2$.
%Note that $q\ge h^{-1}(q^*)$ is equivalent to $H(1/2,q)\le 1-q^*$.
%When $\e = 0$, the lower limits of $U_\e(p^*,q^*)$ are $h^{-1}(p^*)$ and $h^{-1}(q^*)$;
Note that for every $x\in [0,1/2]$, 
$$
h^{-1}(x) \le h^{-1}(\e h\Big(x)+(1-\e)x\Big) \le x.
$$
\begin{figure}[!ht]
  \centering
  \includegraphics[trim=50 20 50 30,clip,width=0.65\textwidth]{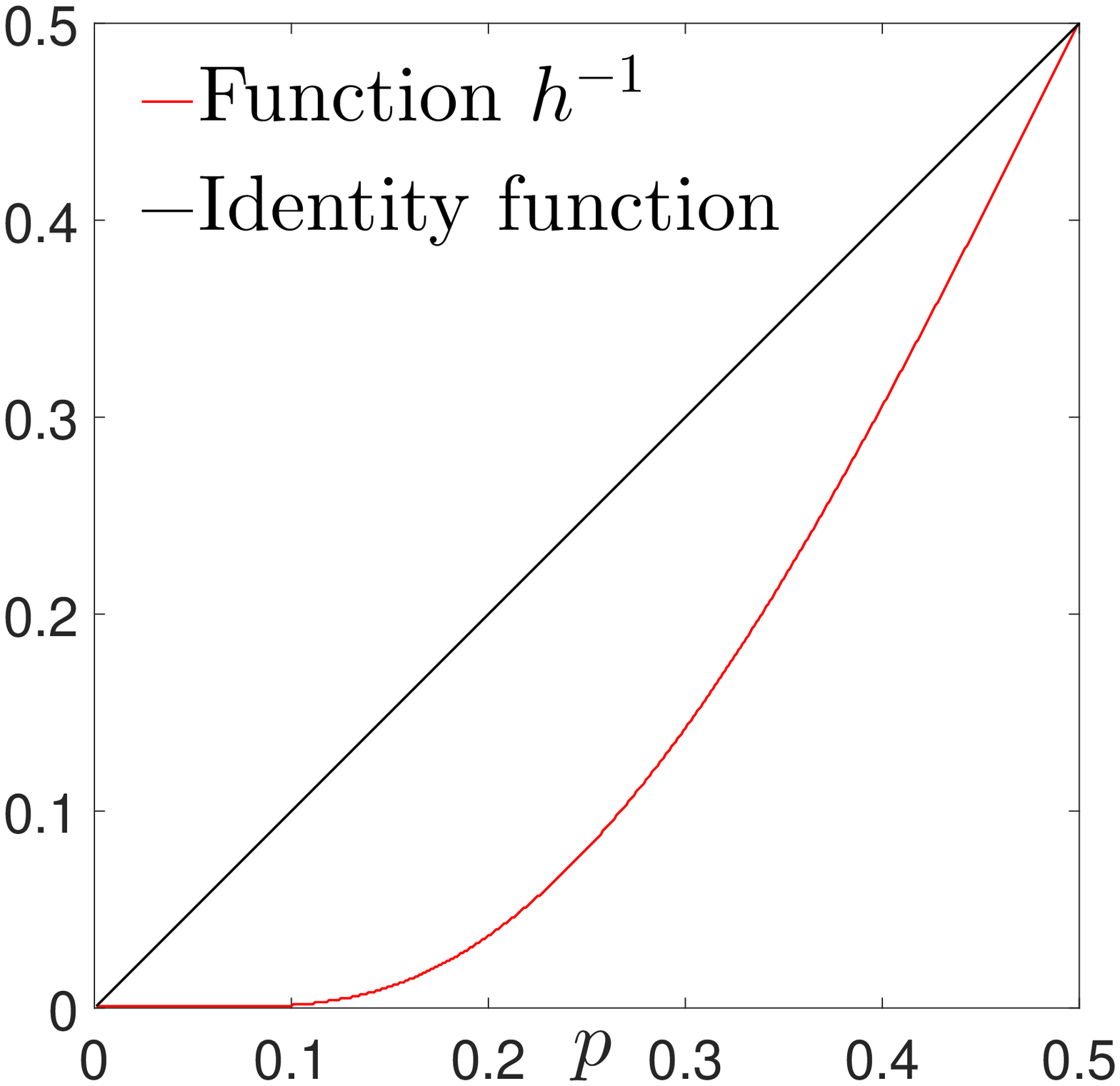}\\
  \caption{Graphs of function $h^{-1}$ and the identity function.}
  \label{fig:neighborhood}
\end{figure}

\begin{lemma}[Contraction of population-level updates]\label{lem: contraction}
Let $\delta \in (0,1/4)$, $\e\in (0,1)$ and $M(\theta')$ be the population-level update of $\theta'$ defined by \eqref{eq: population updates}.
%Fix a pair of nodes $(i,j)$ and let $\theta_{ij}=(W_{ij},P_{ij},Q_{ij})^\tran$, $\theta_{ij}'=(W_{ij}',P_{ij}',Q_{ij}')^\tran$.
Assume that $\theta, \theta'\in R_\delta$, with the set $R_\delta$ defined by \eqref{eq: R delta}, and $(p',q')\in U_\e(p,q)$.
Then
$$
\|M(\theta')-\theta\| \le \frac{30N\|\theta'-\theta\|}{\delta^4}   \exp\Big(- N\delta\e^2 \varphi(p,q)\Big), 
$$
where $\varphi$ is defined by \eqref{eq: function varphi}.  
\end{lemma}

\begin{proof}
The technique used for proving this lemma closely follows \cite{Balakrishnan&Wainwright&Yu}.  
For $t\in[0,1]$, let $\theta_t = (w_t,p_t,q_t)^\tran = \theta + t(\theta'-\theta)$ and define
$g(t,s)=F_{\theta_t}(s)$, where $s$ satisfies \eqref{eq: def of s} and $F_{\theta_t}(s)$ is defined in \eqref{eq: FGL}.
Then $M(w')-w = \E (g(1,s)-g(0,s))$ because $M(w) = w$.
By the mean value theorem, for each $s$ there exists $t_s\in[0,1]$ such that
$$
g(1,s)-g(0,s) = \frac{\partial g(t_s,s)}{\partial t} .
$$

To compute the partial derivative of $g$, note that
$$
F_{\theta}(s) = \frac{1}{1+\exp\left( \log\frac{1-w}{w} + (N-s)\log\frac{1-p}{q} -s\log\frac{1-q}{p}\right)}
=\frac{1}{1+\exp(\ip{Z}{\eta})},
$$
where $\ip{.}{.}$ denotes the inner product, $Z = Z(s) :=(1,N-s,-s)^\tran$ and
$$
\eta = \eta(\theta) := \left(\log\frac{1-w}{w},\log\frac{1-p}{q},\log\frac{1-q}{p}\right)^\tran.
$$
A simple calculation shows that
\begin{eqnarray*}
  \Big|\frac{\partial g(t,s)}{\partial t}\Big|
  &=& \frac{\left|\ip{\psi_t}{\theta'-\theta}\right|}{\Big[ \exp\Big(\frac{1}{2}\ip{Z}{\eta(\theta_t)}\Big) + \exp\Big(-\frac{1}{2}\ip{Z}{\eta(\theta_t)}\Big) \Big]^2} \\
  &\le& \frac{\|\psi_t\|  \|\theta'-\theta\|}
  {\Big[ \exp\Big(\frac{1}{2}\ip{Z}{\eta(\theta_t)}\Big) + \exp\Big(-\frac{1}{2}\ip{Z}{\eta(\theta_t)}\Big) \Big]^2},
\end{eqnarray*}
where
$$
\psi_t = \left( \frac{1}{w_t(1-w_t)},\frac{Np_t-s}{p_t(1-p_t)},\frac{N(1-q_t)-s}{q_t(1-q_t)}\right)^\tran.
$$
Since $s\le N$ and $\theta',\theta\in R_\delta$, it is easy to see that $\|\psi_t\|\le 3N/\delta^2$.
Therefore
\begin{eqnarray}\label{eq: partial der bound}
\nonumber  \E \Big|\frac{\partial g(t_{s},s)}{\partial t} \Big|  &\le&
\frac{3N\|\theta'-\theta\|}{\delta^2} 
\E \frac{1} {\Big[ \exp\Big(\frac{1}{2}\ip{Z(s)}{\eta(\theta_{t_{s}})}\Big) + \exp\Big(-\frac{1}{2}\ip{Z(s)}{\eta(\theta_{t_{s}})}\Big) \Big]^2} \\
&=:& \frac{3N\|\theta'-\theta\|}{\delta^2}  \Phi.
\end{eqnarray}

Let $s_1\sim\text{Binomial}(N,p)$ and $s_2\sim\text{Binomial}(N,1-q)$
be binomial random variables.  
Since $s$ is a mixture of $s_1$ and $s_2$ with weights $1-w$ and $w$, respectively, we have
\begin{eqnarray}
\nonumber  \Phi &\le& (1-w)  \max_{t\in[0,1]}
  \E \frac{1} {\Big[ \exp\Big(\frac{1}{2}\ip{Z(s_1)}{\eta(\theta_t)}\Big) + \exp\Big(-\frac{1}{2}\ip{Z(s_1)}{\eta(\theta_t)}\Big) \Big]^2} \\
\nonumber  &+& w \max_{t\in[0,1]}
  \E \frac{1} {\Big[ \exp\Big(\frac{1}{2}\ip{Z(s_2)}{\eta(\theta_t)}\Big) + \exp\Big(-\frac{1}{2}\ip{Z(s_2)}{\eta(\theta_t)}\Big) \Big]^2}\\
\label{eq: R1 R2}  &=& (1-w) \max_{t\in[0,1]} \E \Phi_1 + w \max_{t\in[0,1]} \E \Phi_2,
\end{eqnarray}
where $\Phi_1$ and $\Phi_2$ denote the corresponding expressions under the expectation.
We now use concentration of $s_1$ and $s_2$ to bound 
$\E \Phi_1$ and $\E\Phi_2$. Note that
\begin{equation}\label{eq: R1 ref}
  \exp\Big(-\ip{Z(s_1)}{\eta(\theta_t)}\Big)
= \frac{1-w_t}{w_t} \left(\frac{1-p_t}{q_t}\right)^N \left(\frac{p_tq_t}{(1-p_t)(1-q_t)}\right)^{s_1} \ge \frac{1-w_t}{w_t}
\end{equation}
if and only if $s_1 \le H(p_t,q_t)N$.
Therefore if $s_1\simeq pN$ is sufficiently smaller than $H(p_t,q_t)N$, then $\exp(-\ip{Z(s_1)}{\eta(\theta_t)})$ grows exponentially in $N$. This implies that $\Phi_1$ is of order $\exp(-N)$ and so is $\E\Phi_1$. 

To make this argument rigorous, 
let $\alpha = \min_{t\in[0,1]} h(p_t)$. Since $(p',q')\in U_\e(p,q)$ and $h$ is monotone, it follows that
\begin{equation}\label{eq: alpha - p star}
\a - p = \min\{ h(p') - p, h(p)-p\} > \e(h(p)-p)> 0.
\end{equation}
Using the monotonicity of $H$, we have
$$
\min_{t\in[0,1]} H(p_t,q_t) \ge \min_{t\in[0,1]} H(p_t,1/2) = \alpha > p.
$$
For $0<\e_0 < \alpha - p$, denote by $\mathcal{E}$ the event $s_1\le(p+\e_0)N$.
By Lemma~\ref{lem: binom tail bound}, we have $\P(\mathcal{E}^c)\le \exp(-2\e_0^2N)$.
When $\mathcal{E}$ occurs,
$$s_1\le (p+\e_0) N < \a N \le N\min_{t\in[0,1]} H(p_t,q_t).$$
This implies that \eqref{eq: R1 ref} holds and 
$s_1\le H(p_t,q_t)N - (\a-p-\e_0)N$ for any $t\in[0,1]$. Since $\theta_t\in R_{\d}$, we have 
$$
\Phi_1 \le \exp\Big(\ip{Z(s_1)}{\eta(\theta_t)}\Big)
\le \frac{1}{\delta}
\left(\max_{t\in[0,1]}\frac{p_tq_t}{(1-p_t)(1-q_t)}\right)^{(\a-p-\e_0\big)N}.
$$
Since the function $x\mapsto x/(1-x)$ is increasing and $p_t,q_t\le 1/2-\delta$ for all $t\in[0,1]$ because $\theta',\theta\in R_\delta$, it follows that
$$\max_{t\in[0,1]} \frac{p_tq_t}{(1-p_t)(1-q_t)}\le \left(\frac{1-2\delta}{1+2\delta}\right)^2.$$
Choose $\e_0=(\alpha-p)/2$. Using the fact that $\P(\mathcal{E}^c)\le \exp(-2\e_0^2N)$ and $\Phi_1\le 1$,  we have
\begin{align*}
  \max_{t\in[0,1]} \E \Phi_1 &\le \exp(-2\e_0^2 N) + \frac{1}{\delta} \exp\Big(-N(\a-p-\e_0) \log(1+4\delta)\Big) \\
  &\le \frac{2}{\delta} \exp\Big(-N \delta(\a-p)^2 \Big) \le  \frac{2}{\delta}  \exp\Big(-N \delta \e^2 \big(h(p)-p\big)^2 \Big).
\end{align*}
Similarly,
$$
\max_{t\in[0,1]} \E \Phi_2 \ \le \ \frac{2}{\delta} \exp\Big(- N \delta\e^2 \big(h(q)-q\big)^2 \Big).
$$
Together with \eqref{eq: partial der bound} and \eqref{eq: R1 R2}, we obtain
\begin{eqnarray}
  \label{eq: Mw bound} |M(w')-w| &\le& \E \Big|\frac{\partial g(t_s,s)}{\partial t} \Big| \\
  \nonumber&\le& \frac{6N\|\theta'-\theta\|}{\delta^3}  \exp\Big(- N\delta\e^2\big(h(p)-p\big)^2\big(h(q)-q\big)^2\Big)=:\Psi.
\end{eqnarray}

It remains to bound $|M(p')-p|$ and $|M(q')-q|$, which boils down to controlling $\E L_{\theta'}(s)$ and $\E G_{\theta'}(s)$. Since $s\le N$, from \eqref{eq: Mw bound}, we have
$$
|\E G_{\theta'}(s)-\E G_{\theta}(s)|
=|\E s(F_{\theta'}(s)-F_{\theta}(s))|
\le N\Psi.
$$ 
Similarly,
$$
|\E L_{\theta'}(s)-\E L_{\theta}(s)|
=|\E G_{\theta'}(s)-\E G_{\theta}(s)|
\le N\Psi.
$$
Combining these inequalities, we obtain
$$
|M(p')-p| \le \frac{3N\Phi}{\delta}, \quad |M(q')-q| \le \frac{3N\Phi}{\delta},
$$
and the claim of Lemma~\ref{lem: contraction} follows.
\end{proof}

\begin{lemma}[Tail bound for binomial distribution]\label{lem: binom tail bound}
If $s\sim \text{Binomial}(N,p)$ then for any $\e\ge 0$,
\begin{eqnarray*}
\Pr{s \le (p-\e)N} &\le& \exp\left(-2\e^2 N\right), \\
\Pr{s \ge (p+\e)N} &\le& \exp\left(-2\e^2N\right).
\end{eqnarray*}
\end{lemma}
\begin{proof}
This is a direct consequence of Hoeffding's inequality.
\end{proof}

%%%%%%%%%% Liza's edits so far

\subsection{Sample-level updates}\label{sec: sample-based updates}
\subsubsection{Preliminaries}
We now turn to the sample-level updates. Let $\hat{c}$ be an estimate of the label assignment $c$.
%In this section we assume that $\hat{c}$ is {\em conditionally independent} of $A^{(m)}$ given $A$.
%This is possible, for example, if we split the data $A^{(m)}$ into two parts, one for estimating $\hat{c}$ and the other for estimating parameters and $A$.
Recall from \eqref{eq: community detection error} that the discrepancy between $\hat{c}$ and $c$ is measured by
$$
\gamma(c,\hat{c}) = \min_{\tilde{c}}\max_{1\le k \le K} \frac{|\{i:\hat{c}_i=k,\tilde{c}_i\neq k\}|+|\{i:\tilde{c}_i=k,\hat{c}_i\neq k\}|}{|\{i: \tilde{c}_i = k\}|},
$$
where the minimum is over all $\tilde{c}$ obtained from $c$ by permuting the labels.
Without loss of generality, assume that the minimum is achieved at $\tilde{c}=c$.
We will focus on a single block (out of $K^2$ blocks)  determined by $c$ and $\hat{c}$.
Fix $k,l\in\{1,...,K\}$ and denote
\begin{equation}\label{eq: J}
  J:=\Big\{(i,j):c_i=k \ \text{and} \ c_j = l\Big\}, \quad \hat{J}:=\Big\{(i,j):\hat{c}_i=k \ \text{and} \ \hat{c}_j = l\Big\}.
\end{equation}
%Thus, $J$ is the set of indices where corresponding node labels are correctly estimated by $\hat{c}$.
By definition of $\gamma(c,\hat{c})$, we have
$$
|J\setminus \hat{J}| \le \gamma^2(c,\hat{c})   |J|, \quad |\hat{J}\setminus J| \le \gamma^2(c,\hat{c})   |J|.
$$

To compare sample-level and population-level updates, for any pair of nodes $(i,j)\in J$, denote (as in Section~\ref{sec: population update})   
$$
w =W_{ij}, \quad p = P_{ij}, \quad q=Q_{ij},\quad \theta = (w,p,q)^\tran.
$$
Also, let $s$ be a mixture of Binomial distributions defined by \eqref{eq: def of s}. Recall the population-level update $M(\theta')$ and its components $M(w')$, $M(p')$ and $M(q')$ in  \eqref{eq: population-level update} and \eqref{eq: population updates}. In the finite sample, instead of taking the expectation of $T_{\theta''|\theta'}$, we compute the average of $T_{\theta''|\theta'}$ over all entries within the block $\hat{J}$. The sample-level update is then the maximizer of this average:
$$
\hat{M}(\theta') = \argmax_{\theta''} \frac{1}{|\hat{J}|} \sum_{(i,j)\in \hat{J}} T_{\theta''|\theta'}(S_{ij}),
$$

To compute $\hat{M}(\theta')$, denote  
$$
\hat{\E} F_{\theta'} = \frac{1}{|\hat{J}|} \sum_{(i,j)\in \hat{J}} F_{\theta'}(S_{ij}), \qquad
\hat{\E} G_{\theta'} = \frac{1}{|\hat{J}|} \sum_{(i,j)\in \hat{J}} G_{\theta'}(S_{ij}), \qquad 
\hat{\E} L_{\theta'} = \frac{1}{|\hat{J}|} \sum_{(i,j)\in \hat{J}} L_{\theta'}(S_{ij}),
$$
where $F_{\theta'}$, $G_{\theta'}$ and $L_{\theta'}$ are defined in \eqref{eq: FGL}.
Then similar to \eqref{eq: population updates}, the components of $\hat{M}(\theta')$, which we denote by $\hat{M}(w')$, $\hat{M}(p')$ and $\hat{M}(q')$, can be computed by
\begin{equation}\label{eq: sample updates}
\hat{M}(w') = \hat{\E} F_{\theta'},
\quad \hat{M}(p') = \frac{\hat{\E} L_{\theta'}}{ N- N  \hat{\E} F_{\theta'}},
\quad \hat{M}(q') = \frac{N  \hat{\E}F_{\theta'} - \hat{\E} G_{\theta'}}{N  \hat{\E} F_{\theta'}}.
\end{equation}

\subsubsection{Concentrations of sample-level updates}
We first prove uniform bounds for $\hat{\E} F_{\theta'}-\E F_{\theta'}(s)$, $\hat{\E} G_{\theta'}-\E G_{\theta'}(s)$ and $\hat{\E} L_{\theta'}-\E L_{\theta'}(s)$ in Lemmas~\ref{lem: sample updates}~and~\ref{lem: sample updates large N}. Combined with Lemma~\ref{lem: contraction}, this yields a uniform bound for $\hat{M}(\theta')-M(\theta')$ in Corollary~\ref{cor: sample updates}. Finally, Lemma~\ref{lem: validity of initial estimates} shows that our initial parameter estimates belong to a desired neighborhood of the true parameter $\theta$. The convergence of the sample-level updates then follows from the contraction of the population updates (Lemma~\ref{lem: contraction}), the uniform bound between the sample-level updates and the population updates (Corollary~\ref{cor: sample updates}), and the accuracy of the initial parameter estimates (Lemma~\ref{lem: validity of initial estimates}).

\begin{lemma}[Concentration of sample updates with small $N$]\label{lem: sample updates}
Let $\delta\in(0,1/4)$ and $s$ be a mixture of Binomial distributions defined by \eqref{eq: def of s}. 
Then for any $r\ge 0$ the following hold with probability at least $1-e^{-r}$:  
\begin{eqnarray*}
\sup_{\theta'\in R_\d} |\hat{\E} F_{\theta'}-\E F_{\theta'}(s)| &\le& 4\gamma^2(\hat{c},c) + 50r \log^2\left(\frac{1}{\delta}-1\right)  \sqrt{\frac{N}{|J|}},\\
\sup_{\theta'\in R_\d} | \hat{\E} G_{\theta'} - \E  G_{\theta'}(s)|
 &\le& 4N\gamma^2(\hat{c},c) + 50rN
       \log^2\left(\frac{1}{\delta}-1\right) \sqrt{\frac{N}{|J|}}, 
\end{eqnarray*}
where the supremum is taken over $R_\delta$ defined in \eqref{eq: R delta}.  
\end{lemma}
Note that $|\hat{\E} L_{\theta'} - \E  L_{\theta'}(s)| = | \hat{\E} G_{\theta'} - \E  G_{\theta'}(s)|$, therefore the second inequality of Lemma~\ref{lem: sample updates} also provides a bound for $\sup_{\theta'\in R_\d} | \hat{\E} L_{\theta'} - \E  L_{\theta'}(s)|$.
\begin{proof}
We begin with bounding $\hat{\E} F_{\theta'}-\E F_{\theta'}(s)$. Denote
\begin{equation}\label{eq: def of Psi}
\Psi :=  \frac{1}{|J|} \sum_{(i,j)\in J} F_{\theta'}(S_{ij}).
\end{equation}
Then $\Psi$ is the average of $F_{\theta'}(S_{ij})$ over $J$ while $\hat{M}(w')$ is the average of $F_{\theta'}(S_{ij})$ over $\hat{J}$.
Using the fact that $|F_{\theta'}(S_{ij})|\le 1$ and the definition of $\gamma^2(\hat{c},c)$, we have
\begin{equation}\label{eq: error between hat M and Psi}
  |\hat{\E} F_{\theta'}-\Psi| \le \left|\frac{1}{|\hat{J}|} - \frac{1}{|J|}\right|  |\hat{J}\cap J| + \frac{|\hat{J}\setminus J|}{|\hat{J}|} + \frac{|J\setminus\hat{J}|}{|J|} \le 4\gamma^2(\hat{c},c).
\end{equation}
We now focus on bounding $\Psi-\E F_{\theta'}(s)$.
Condition on $J$, $S_{ij}$ are i.i.d. copies of a mixture of binomial distributions $w \text{Binomial}(N,1-q)+(1-w) \text{Binomial}(N,p)$. Let $\l>0$ be a positive scalar and $\e_{ij}$ be independent symmetric Bernoulli random variables, also independent of $S_{ij}$.
By symmetrization (see e.g. \cite[Theorem 2.1]{Koltchinskii2011}), we have
\begin{equation}\label{eq: symmetrization}
  \E \exp\left( \l \sup_{{\theta'}\in R_\d}| \Psi - \E F_{\theta'}(s)|\right) \le \E \exp \left( \frac{2\l}{|J|}
  \sup_{{\theta'}\in R_\d} \left|\sum_{(i,j)\in J} \e_{ij} \left( F_{\theta'}(S_{ij}) -\frac{1}{2} \right)\right|\right).
\end{equation}
Note that
$ F_{\theta'}(S_{ij}) = 1/\big(1+\exp(Z_{ij}^\tran \eta(\theta'))\big)$,
where
\begin{equation}\label{eq: Z eta}
Z_{ij} := (1,N-S_{ij},-S_{ij})^\tran, \quad \eta(\theta') := \left(\log\frac{1-w'}{w'},\log\frac{1-p'}{q'},\log\frac{1-q'}{p'}\right)^\tran.
\end{equation}
Since $t\mapsto 1/(1+e^t)-1/2$ is Lipschitz with constant one
and $\|\eta(\theta')\|\le 2 \log(1/\delta-1)$ because $\theta'\in R_\d$,
using \eqref{eq: symmetrization} and Talagrand's comparision theorem (see e.g. \cite[Theorem~4.12]{Ledoux&Talagrand1991}), we obtain
\begin{eqnarray}
\nonumber\E \exp\left( \l \sup_{\theta'\in R_\d}| \Psi - \E F_{\theta'}(s)|\right) &\le&
  \E \exp\left( \frac{4\l}{|J|}
  \sup_{\theta'\in R_\d}\Big|\sum_{(i,j)\in J} \e_{ij} Z_{ij}^\tran \eta(\theta')\Big|\right)\\
 \nonumber&\le& \E \exp\left( \frac{4\l\log\big(\frac{1}{\delta}-1\big)}{|J|} \Big\|
 \sum_{(i,j)\in J} \e_{ij} Z_{ij} \Big\|\right)\\
 \nonumber &\le& \sum_{\kappa\in\{-1,1\}^3} \E \exp\left( \frac{4\l\log\big(\frac{1}{\delta}-1\big)}{|J|}
 \sum_{(i,j)\in J} \e_{ij} Z_{ij}^\tran \kappa \right)\\
 \nonumber &=& \sum_{\kappa\in\{-1,1\}^3} \prod_{(i,j)\in J}
 \E \exp\left( \frac{4\l\log\big(\frac{1}{\delta}-1\big)}{|J|}  \e_{ij} Z_{ij}^\tran \kappa \right).
\end{eqnarray}
Since $\e_{ij} Z_{ij}^\tran \kappa$ are sub-Gaussian random variables with sub-Gaussian norm at most $\sqrt{3N}$,
we have $\E \exp(t\e_{ij} Z_{ij}^\tran \kappa)\le \exp(3t^2N)$ for every $t\ge 0$.
Therefore
\begin{eqnarray*}
\E \exp\left( \l  \sup_{\theta'\in R_\d}| \Psi - \E F_{\theta'}(s)|\right)
&\le& 8 \exp\left(  \frac{48N\l^2\log^2\big(\frac{1}{\delta}-1\big)}{|J|} \right).
\end{eqnarray*}
Using Markov inequality, we conclude that with probability at least $1-e^{-r}$ the following holds:
\begin{eqnarray*}\label{eq: Psi bound}
\sup_{\theta'\in R_\d} |\Psi-\E F_{\theta'}(s)| \le 50r \log^2\left(\frac{1}{\delta}-1\right)  \sqrt{\frac{N}{|J|}}.
\end{eqnarray*}
Therefore using \eqref{eq: error between hat M and Psi} and a triangle inequality, we obtain that with probability at least $1-e^{-r}$:
\begin{equation*}
  \sup_{\theta'\in R_\d} |\hat{\E} F_{\theta'}-\E F_{\theta'}(s)| \ \le \ 4\gamma^2(\hat{c},c) + 50r \log^2\left(\frac{1}{\delta}-1\right)  \sqrt{\frac{N}{|J|}}
\end{equation*}

It remains to bound $\hat{\E} G_{\theta'}$, which is similar to $\hat{\E} F_{\theta'}$ except that $G_{\theta'}$ contains an additional factor $s$. Proceeding the proof in the same way as for $\hat{\E} F_{\theta'}$, and bound $s$ by $N$ when necessary, we obtain that with probability at least $1-e^{-r}$, the following holds
\begin{equation*}
  \sup_{\theta'\in R_\d}
  \Big| \hat{\E} G_{\theta'}
   - \E  G_{\theta'}(s)\Big|
\ \le \ 4N\gamma^2(\hat{c},c) + 50rN \log^2\left(\frac{1}{\delta}-1\right)  \sqrt{\frac{N}{|J|}}.
\end{equation*}
The proof is complete.
\end{proof}

%The purpose of the next lemma (Lemma~\ref{lem: sample updates large N}) is two-fold. First, it provides a lower bound on the denominators of $M(p')$ and $M(q')$; together with Lemma~\ref{lem: sample updates} this allows us to bound $\hat{M}(p')-M(p')$ and $\hat{M}(q')-M(q')$. 
The following lemma provides alternative bounds to the bounds in Lemma~\ref{lem: sample updates} when $N$ is large. Note that the upper bounds of     
Lemma~\ref{lem: sample updates} contain the factor $\sqrt{N/|J|}$; they become uninformative when $N$ is larger than $|J|$. This is an artifact of our proof as we use Talagrand's comparison theorem. Lemma~\ref{lem: sample updates large N} shows that when $N$ is large, we can directly compare $\hat{M}(\theta')$ and the true parameter $\theta$ and effectively remove the factor $\sqrt{N}$.   
  
\begin{lemma}[Sample updates with large $N$]\label{lem: sample updates large N}
Let $\delta \in (0,1/4)$ and $\e\in (0,1)$. Assume that $\theta', \theta\in R_\delta$ and $(p',q')\in U_\e(p,q)$.
Then there exists a constant $C>0$ such that for any $r>0$ the following holds with probability at least $1-\exp(-r)$:
\begin{eqnarray*}
\sup_{\theta'\in R_\d} |\hat{\E} F_{\theta'}-w| &\le& \frac{C}{\d} \big(1+\d\big)^{-N\e^2\varphi(p,q)}+\frac{Cr}{\sqrt{\d|J|}} + 4\gamma^2(\hat{c},c),\\
\sup_{\theta'\in R_\d} |\hat{\E} G_{\theta'}-Npw| &\le& \frac{CN}{\d} \big(1+\d\big)^{-N\e^2\varphi(p,q)}+\frac{CNr}{\sqrt{\d|J|}} + 4N\gamma^2(\hat{c},c).
\end{eqnarray*}
%Moreover, 
%\begin{eqnarray*}
%\Big| {\E} \frac{s X_{\theta'}(s) }{X_{\theta'}(s)+Y_{\theta'}(s)} - Npw\Big| &\le& \frac{CN}{\d}  \Big(1+\d\Big)^{-N\e^2\big[h(p)-p\big]^2\big[h(q)-q\big]^2},\\
%\sup_{\theta'\in R_\d} \Big| \hat{\E} \frac{s X_{\theta'}(s) }{X_{\theta'}(s)+Y_{\theta'}(s)} - Npw\Big|  &\le& \frac{CN}{\d} \big(1+\d\big)^{-N\e^2[h(p)-p]^2 [h(q)-q]^2}+\frac{CNr}{\sqrt{\d|J|}} + 4N\gamma^2(\hat{c},c).
%\end{eqnarray*}
\end{lemma}

\begin{proof}

\medskip
We show the first inequality; the second inequality is proved using a similar argument. From \eqref{eq: def of Psi} and \eqref{eq: error between hat M and Psi} we have
\begin{equation}\label{eq: change bound of hatmw to Psi}
|\hat{M}(w')-w| \le |\hat{M}(w')-\Psi| + |\Psi - w|\le 4\gamma^2(\hat{c},c) + |\Psi - w|.  
\end{equation}
Therefore it is enough to bound $|\Psi-w|$. 
Denote by $J_0$ the set of indices $(i,j)\in J$ such that $A_{ij}=0$ and by $J_1$ the set of indices $(i,j)\in J$ such that $A_{ij}=1$. Then $S_{ij}\sim{\rm Binomial}(N,p)$ if $(i,j)\in J_0$ and $S_{ij}\sim{\rm Binomial}(N,1-q)$ if $(i,j)\in J_1$. By Lemma~\ref{lem: binom tail bound}, for $r>0$ we have
\begin{equation}\label{eq: bound on J_1}
\Pr{\big||J_1|-|J|w\big|> r\sqrt{|J|}} \le 2\exp(-2r^2).
\end{equation}
Note that
\begin{equation}\label{eq: decomposition of Psi}
\Psi =  \frac{1}{|J|} \sum_{(i,j)\in J_0} F_{\theta'}(S_{ij})
+ \frac{1}{|J|} \sum_{(i,j)\in J_1} F_{\theta'}(S_{ij}) =: \Psi_0+\Psi_1.
\end{equation}

We first show that $\Psi_1$ is close to $w$. Note that
$ F_{\theta'}(S_{ij}) = 1/\big(1+\exp(Z_{ij}^\tran \eta(\theta'))\big)$,
where
\begin{equation*}
Z_{ij} := (1,N-S_{ij},-S_{ij})^\tran, \quad \eta(\theta') := \left(\log\frac{1-w'}{w'},\log\frac{1-p'}{q'},\log\frac{1-q'}{p'}\right)^\tran.
\end{equation*} 
Condition on $A$, by Lemma~\ref{lem: binom tail bound}, for any $\e_1>0$ and $(i,j)\in J_1$ the following holds
\begin{equation}\label{eq: conditional prob of Sij}
\Pr{S_{ij}<(1-q-\e_1)N} \le \exp(-2\e_1^2 N).
\end{equation}
Therefore  with conditional probability at least $1-\exp(-2\e_1^2 N)$ we have
\begin{eqnarray}
\nonumber Z_{ij}^\tran\eta(\theta')  
&=& \log\frac{1-w'}{w'} + \log \frac{(1-p')(1-q')}{p'q'} \big[N H(p',q')-S_{ij}\big]\\
&\le& \label{eq: Z eta 1}
\log\frac{1-w'}{w'} + N\log \frac{(1-p')(1-q')}{p'q'}\big[H(p',q')-(1-q-\e_1)\big]. 
\end{eqnarray} 
Here $H$ is the function defined in \eqref{eq: function H}.
Since $(p',q')\in U_\e(p,q)$, it follows that $h(q')\ge \e h(q)+(1-\e)q$. Using the fact that $H(p',q')$ is increasing in $p'$ and $h$ is increasing, we have
$$
H(p',q')-(1-q) \le H(1/2,q')-1+q=-h(q')+q \le -\e(h(q)-q) < 0.  
$$
Note that if $\d\le t\le 1/2-\d$ then $1+4\d\le (1-t)/t\le (1-\d)/\d$. Therefore for $\theta'\in R_\d$, using $\e_1=\e(h(q)-q)/2$, we obtain
\begin{equation*}
Z_{ij}^\tran\eta(\theta') \le \log\frac{1-\d}{\d} - N\e (h(q)-q) \log(1+4\d).
\end{equation*}
Denote by $I_1$ the set of indices $(i,j)\in J_1$ such that $S_{ij} \ge (1-q-\e_1)N$.
Using inequalities $0\le F_{\theta'}(S_{ij}) \le 1$ and $1/(1+x)\ge 1-x$ for $x\ge 0$, this implies
\begin{equation}\label{eq: bound sum over I1}
|I_1| \left(1-\frac{1}{\d} \big(1+4\d\big)^{-N\e(h(q)-q)} \right)\le \sum_{(i,j)\in I_1} F_{\theta'}(S_{ij}) \le |J_1|.
\end{equation}

To lower bound $|I_1|$, note that $|I_1|$ is a sum of independent Bernoulli random variables with success probabilities at least $1-\exp(-2\e_1^2N)$ by \eqref{eq: conditional prob of Sij}. By Lemma~\ref{lem: binom tail bound} and condition on $A$, we have 
$$\Pr{|I_1|\ge|J_1|\left(1-\exp(-2\e_1^2N)-r/\sqrt{\d|J|}\right)}\ge 1-\exp\left(-2r^2|J_1|/(\d|J|)\right).$$
Using \eqref{eq: bound on J_1}, assumption $|J|\ge 4r^2/\d^2$ and $\e_1=\e(h(q)-q)/2$, we obtain that with probability at least $1-3\exp(-r^2)$ the following holds:  
$$
\frac{|I_1|}{|J|}\ge w-\exp\left(\frac{-N\e^2(h(q)-q)^2)}{2}\right)-\frac{2r}{\sqrt{\d|J|}}.
$$
It then follows from \eqref{eq: bound on J_1} and \eqref{eq: bound sum over I1} that with probability at least $1-3\exp(-r^2)$ the following holds:
\begin{equation}\label{eq: bound of Psi1}
w-\frac{2}{\d} \big(1+\d\big)^{-N\e^2(h(q)-q)^2}-\frac{2r}{\sqrt{\d|J|}}
\le \sup_{\theta\in R_\d} \Psi_1 \le w+\frac{r}{\sqrt{|J|}}.
\end{equation}

Similarly, with probability at least $1-3\exp(-r^2)$ we have
\begin{equation}\label{eq: upper bound on Psi zero}
\sup_{\theta'\in R_\d}\Psi_0\le
\frac{2}{\d} \big(1+\d\big)^{-N\e^2(h(p)-p)^2}+\frac{r}{\sqrt{\d|J|}}.
\end{equation}
From \eqref{eq: bound of Psi1}, \eqref{eq: upper bound on Psi zero} and using a triangle inequality, we obtain that with probability at least $1-6\exp(-r^2)$:
\begin{equation*}
\sup_{\theta'\in R_\d}|\Psi-w| \le \frac{4}{\d} \big(1+\d\big)^{-N\e^2[h(p)-p]^2 [h(q)-q]^2}+\frac{4r}{\sqrt{\d|J|}}.
\end{equation*} 
Together with \eqref{eq: change bound of hatmw to Psi} this provides a bound on $|\hat{M}(w')-w|$.
\end{proof}

\begin{corollary}[Sample updates]\label{cor: sample updates}
Let $\delta \in (0,1/4)$ and $\e\in (0,1)$. Then there exist constants $C_1,C_2>0$ such that for any $r>0$ the following holds with probability at least $1-\exp(-r)$.
Assume that $\theta', \theta\in R_\delta$, $(p',q')\in U_\e(p,q)$ and
$$
N\ge \frac{C_1\log(1/\d)}{\e^2\big[h(p)-p\big]^2\big[h(q)-q\big]^2}.
$$
Then
\begin{eqnarray*}
\sup_{\theta'\in R_\d}\|\hat{M}(\theta')-M(\theta')\| &\le& \frac{ C_2 r \Phi + 8\gamma^2(\hat{c},c)}{\d},
\end{eqnarray*}
where 
$$
\Phi := \min\left\{\log^2\left(\frac{1}{\delta}-1\right) \sqrt{\frac{N}{|J|}}, \ \ \frac{1}{\d} \big(1+\d\big)^{-N\e^2[h(p)-p]^2 [h(q)-q]^2}+\frac{1}{\sqrt{\d|J|}}\right\}.
$$
\end{corollary}
\begin{proof}
Recall the components of $\hat{M}(\theta')$ and $M(\theta')$ in \eqref{eq: population updates} and \eqref{eq: sample updates}. Lemma~\ref{lem: sample updates} and Lemma~\ref{lem: sample updates large N} show that $\hat{M}(w')$ concentrates around $M(w')$; they also show that numerators and denominators of $\hat{M}(p')$ and $\hat{M}(q')$ concentrate around that of $M(p')$ and $M(q')$. To obtain a bound for $\|\hat{M}(\theta')-M(\theta')\|$, it remains to bound the denominators of $M(p')$ and $M(q')$ away from zero. That is done by the help of Lemma~\ref{lem: contraction}.
\end{proof}

%\subsection{Initial parameter estimates}
\medskip
We now show that initial parameter estimates of our algorithm belong to a desired neighborhood of the true parameters; this allows us to establish consistency of our algorithm.
Denote by $\theta_0=(w_0,p_0,q_0)^\tran$ the initial value of $\theta$ taken by our algorithm, and
\begin{equation}\label{eq: def of I}
I=\big\{(i,j)\in J: S_{ij} < N/2\big\}, \quad \hat{I}=\big\{(i,j)\in \hat{J}: S_{ij} < N/2\big\}.
\end{equation}
%and by $\tilde{I}=\hat{J}_{\hat{c}}\setminus I$ the complement of $I$ in $\hat{J}_{\hat{c}}$.
From Section~\ref{subsec: em-type algorithm} we have
\begin{equation}\label{eq: intial parameter estimates}
w_0 = \frac{|\hat{J}\setminus\hat{I}|}{|\hat{J}|}, \qquad
p_0 = \frac{\sum_{(i,j)\in \hat{I}}S_{ij}}{N|\hat{I}|}, \qquad
q_0 = \frac{\sum_{(i,j)\in \hat{J}\setminus\hat{I}} (N-S_{ij}) }{N|\hat{J}\setminus\hat{I}|}.
\end{equation}
For $x\in[0,1/2]$, denote
\begin{equation}\label{eq: function ell}
  \phi(x) := x-h^{-1}\left(\frac{x}{2}+\frac{h(x)}{2}\right),
\end{equation}
where $h$ is the function defined in \eqref{eq: function H}.
Note that $\phi(x)\ge 0$ and $\phi(x)=0$ if and only if $x=0$ or $x=1/2$.
Recall the definition of $U_\e(p,q)$ in \eqref{eq: neighborhood of p* q*} and definition of $R_\delta$ in \eqref{eq: R delta}.
\begin{lemma}[Validity of initial parameter estimates]\label{lem: validity of initial estimates}
Assume that $\theta\in R_\delta$ and the following conditions hold for some constant $C>0$:
\begin{eqnarray*}
  N &\ge& \frac{C}{\delta^2}\max\left\{\log\frac{1}{\delta^2}, \log\frac{1}{\delta\phi(p)}, \ \log\frac{1}{\delta\phi(q)}\right\}, \\
  |J| &\ge& \frac{C}{\delta^3}
  \max\left\{\frac{r^2}{\delta^2},\frac{r^2}{\phi^2(p)}, \frac{r^2}{\phi^2(q)}\right\}, \\
  \gamma^2(\hat{c},c) &\le& \frac{\delta}{C}  \max\left\{\delta, \phi(p), \phi(q) \right\},
\end{eqnarray*}
where $\phi$ is the function defined in \eqref{eq: function ell}.
Then $\theta_0\in R_{\delta}$ and $(p_0,q_0)\in U_{1/2}(p,q)$ with probability at least $1-25\exp(-r^2)$.
\end{lemma}
\begin{proof}
%Let $\tilde{J}=\hat{J}_{\hat{c}}\setminus J$ be the complement of $J$ in  $\hat{J}_{\hat{c}}$.
We partition $J$ as $J = J_0\cap J_1$, where
$$
J_0 = \{(i,j)\in J: A_{ij}= 0\}, \qquad J_1 = \{(i,j)\in J: A_{ij}=1\}.
$$
Similarly, we partition $\hat{J}$ as $\hat{J}=\hat{J}_0\cup\hat{J}_1$, where $\hat{J}_0$ and $\hat{J}_1$ are sets of indices $(i,j)\in \hat{J}$ where $A_{ij}=0$ and $A_{ij}=1$, respectively. Also, denote $I_c = J\setminus I$ and $\hat{I}_c=\hat{J}\setminus\hat{I}$, where $I$ and $\hat{I}$ are defined by \eqref{eq: def of I}.

\medskip
\noindent{\em Edge probability estimate $w_0$.}
Recall that $w_0 = |\hat{I}_c|/|\hat{J}|$, where $\hat{I}_c$ is the set of $(i,j)\in \hat{J}$ where $S_{ij}\ge N/2$.
Since
$$
\hat{I}_c =  \big(\hat{I}_c \cap J_1\big) \cup \big(\hat{I}_c\cap J_0\big) \cup \big(\hat{I}_c\setminus J\big), \quad \hat{I}_c\cap J_0 \subseteq I_c\cap J_0,
$$
it follows that $| \hat{I}_c| \le |J_1|+|I_c\cap J_0|+|\hat{J}\setminus J| $.
Moreover, since 
$$
\hat{I}_c\cap J_1 \subseteq I_c\cap J_1, \quad
(I_c\cap J_1)\setminus (\hat{I}_c\cap J_1) \subseteq J\setminus \hat{J},
$$
we have $|\hat{I}_c| \ge |I_c\cap J_1|-|J\setminus \hat{J}|$.
Together with the definition of $\gamma^2(\hat{c},c)$, we obtain
\begin{equation}\label{eq: w bounds detection error}
\frac{|I_c \cap J_1|-\gamma^2(\hat{c},c)|J|}{|J|+\gamma^2(\hat{c},c)|J|}
\le w_0 \le \frac{|J_1| + |I_c\cap J_0| + \gamma^2(\hat{c},c)|J|}{|J|-\gamma^2(\hat{c},c)|J|}.
%  \frac{|\tilde{I}\cap J|}{|J|} - \gamma^2(\hat{c},c) \ \le \ w_0 \ \le  \ \frac{|\tilde{I}\cap J|}{|J|} + \gamma^2(\hat{c},c).
\end{equation}
We now estimate the terms in \eqref{eq: w bounds detection error}. Note that $S_{ij}$ follows a mixture of Binomial distributions with parameter $\theta$ for every $(i,j)\in J$.
By Lemma~\ref{lem: binom tail bound}, we have
\begin{equation}\label{eq: |J_1| upper bound}
  \Pr{|J_1|\le w|J| + r|J|^{1/2}} \ge 1-\exp(-2r^2).
\end{equation}

We now upper bound $|I_c\cap J_0|$.
Condition on $A$, for every $(i,j)\in J_0$, by Lemma~\ref{lem: binom tail bound} we have
$$
\Pr{S_{ij}\ge N/2} \le \exp(-2\delta^2 N).
$$
Since $|I_c\cap J_0|$ is a sum of Bernoulli random variables with success probability at most $\exp(-2\delta^2 N)$,
by Lemma~\ref{lem: binom tail bound},
$$
\P\Big\{ |I_c\cap J_0| \le (\exp(-2\delta^2 N)+\e)|J_0| \Big\} \ge 1-\exp(-2\e^2 |J_0|).
$$
Also, using Lemma~\ref{lem: binom tail bound} and assumption $|J|\ge 4r^2/\delta^2$, we have
\begin{equation*}\label{eq: crude lower bound for |J_0|}
  \P\Big\{|J_0|\ge \delta |J|/2\Big\} \ge 1-\exp(-2r^2).
\end{equation*}
Setting $\e = r/(\delta|J|)^{1/2}$ and using $|J_0|\le |J|$, we obtain
\begin{equation}\label{eq: upper bound of I tilde cap J_0 complete}
  \Pr{ |I_c\cap J_0| \le \exp(-2\delta^2 N)|J| + r(|J|/\delta)^{1/2} } \ge 1-2\exp(-r^2).
\end{equation}

We lower bound $|I_c\cap J_1|$ in a similar way.
Condition on $A$, by Lemma~\ref{lem: binom tail bound}, for each entry $(i,j)\in\cap J_1$ we have
$$
\Pr{S_{ij}\ge N/2} \ge 1 - \exp(-2\delta^2 N).
$$
Since $|I_c\cap J_1|$ is a sum of Bernoulli random variables with success probability at least $1-\exp(-2\delta^2 N)$, by Lemma~\ref{lem: binom tail bound},
$$
\Pr{|I_c\cap J_1| \ge \big(1-\exp(-2\delta^2 N)-\e\big)|J_1|} \ge 1-\exp(-2\e^2|J_1|).
$$
Also, by Lemma~\ref{lem: binom tail bound},
\begin{equation}\label{eq: lower bound for |J_1|}
  \P\Big\{|J_1|\ge w^*|J| - r|J|^{1/2} \Big\} \ge 1-\exp(-2r^2).
\end{equation}
Setting $\e = r/(\delta|J|)^{1/2}$ and using assumption $|J| \ge 4r^2/\delta^2$, we obtain
\begin{equation}\label{eq: upper bound of I tilde cap J_1 complete}
  \Pr{ |I_c\cap J_1| \ge w|J| -\exp(-2\delta^2 N)|J| -2 r(|J|/\delta)^{1/2} } \ge 1-2\exp(-r^2).
\end{equation}

We are now ready to bound $w_0$ using \eqref{eq: w bounds detection error} and subsequent estimates of terms in \eqref{eq: w bounds detection error}.
Using \eqref{eq: |J_1| upper bound}, \eqref{eq: upper bound of I tilde cap J_0 complete}, \eqref{eq: upper bound of I tilde cap J_1 complete} and assumptions on $N$, $|J|$ and $\gamma^2(\hat{c},c)$, we obtain
\begin{equation}\label{eq: bound on w_0 final}
  \Pr{ \frac{\delta}{2} \le w_0 \le 1-\frac{\delta}{2}} \ge 1-5\exp(-r^2).
\end{equation}

\medskip
\noindent{\em False positive estimate $p_0$.}
Recall that
$$p_0 = \frac{\sum_{(i,j)\in \hat{I}}S_{ij}}{N|\hat{I}|},$$
where $\hat{I}$ is the set of $(i,j)\in \hat{J}$ such that $S_{ij}<N/2$.
Note that
$$
\hat{I} =  \big(\hat{I} \cap J_0\big) \cup \big(\hat{I} \cap J_1\big) \cup \big(\hat{I} \setminus J\big).
$$
Using the partition of $\hat{I}$, the bound $S_{ij}\le N$ and the definition of $\gamma^2(\hat{c},c)$, we obtain
\begin{equation}\label{eq: first two way bounds for p_0}
\frac{\sum_{(i,j)\in J_0} S_{ij}-N|I_c\cap J_0|-N\gamma^2(\hat{c},c)|J|}{N\big(|J_0|+|I\cap J_1|+\gamma^2(\hat{c},c)|J|\big)}\le p_0
\le \frac{\sum_{(i,j)\in J_0} S_{ij}+N|I\cap J_1| + N\gamma^2(\hat{c},c)|J|}{N|I\cap J_0|-N\gamma^2(\hat{c},c)|J|}
\end{equation}
We now estimate the terms in \eqref{eq: first two way bounds for p_0}.
%Note that $S_{ij}\sim f_{\theta^*}$ follows a mixture of Binomial distributions with parameter $\theta^*$ for every $(i,j)\in J = J_0\cap J_1$.
By Lemma~\ref{lem: binom tail bound}, we have
\begin{equation}\label{eq: |J_0| upper bound}
  \Pr{|J_0|\le (1-w)|J| + r|J|^{1/2}} \ge 1-\exp(-2r^2).
\end{equation}

To upper bound $|I\cap J_1|$, we first condition on $A$. By Lemma~\ref{lem: binom tail bound}, for each $(i,j)\in\cap J_1$ we have
$$
\Pr{S_{ij}<N/2} \le \exp(-2\delta^2 N).
$$
Since $|I\cap J_1|$ is a sum of Bernoulli random variables with success probability at most $\exp(-2\delta^2 N)$, by Lemma~\ref{lem: binom tail bound},
$$
\Pr{|I\cap J_1| \le \big(\exp(-2\delta^2 N)+\e\big)|J_1|} \ge 1-\exp(-2\e^2|J_1|).
$$
Also, using Lemma~\ref{lem: binom tail bound} and assumption  $|J|\ge 4r^2/\delta^2$, we have
\begin{equation*}\label{eq: lower bound on |J_0|}
  \P\Big\{|J_1|\ge \delta |J|/2\Big\} \ge 1-\exp(-2r^2).
\end{equation*}
Setting $\e = r/(\delta|J|)^{1/2}$, we obtain
\begin{equation}\label{eq: upper bound of I cap J_1 complete}
  \Pr{ |I\cap J_1| \le \exp(-2\delta^2 N)|J| + r(|J|/\delta)^{1/2} } \ge 1-2\exp(-r^2).
\end{equation}

We now lower bound $|I\cap J_0|$.
Condition on $A$, by Lemma~\ref{lem: binom tail bound}, for each entry $(i,j)\in\cap J_0$ we have
$$
\Pr{S_{ij} < N/2} \ge 1 - \exp(-2\delta^2 N).
$$
Since $|I\cap J_0|$ is a sum of Bernoulli random variables with success probability at least $1-\exp(-2\delta^2 N)$, by Lemma~\ref{lem: binom tail bound},
$$
\Pr{|I\cap J_0| \ge \big(1-\exp(-2\delta^2 N)-\e\big)|J_0|} \ge 1-\exp(-2\e^2|J_0|).
$$
By Lemma~\ref{lem: binom tail bound},
\begin{equation}\label{eq: lower bound for |J_0|}
  \P\Big\{|J_0|\ge (1-w)|J| - r|J|^{1/2} \Big\} \ge 1-\exp(-2r^2).
\end{equation}
Setting $\e = r/(\delta|J|)^{1/2}$ and using assumption $|J| \ge 4r^2/\delta^2$, we obtain
\begin{equation}\label{eq: upper bound of I cap J_0 complete}
  \Pr{ |I\cap J_0| \ge (1-w^*)|J| -\exp(-2\delta^2 N)|J| -2 r(|J|/\delta)^{1/2} } \ge 1-2\exp(-r^2).
\end{equation}

It remains to control $\sum_{(i,j)\in J_0} S_{ij}$.
Condition on $A$, by Bernstein's inequality, we have
\begin{equation*}
\Pr{\Big| \sum_{(i,j)\in J_0}  S_{ij}- |J_0| N p \Big| > t} \le 2 \exp\left(\frac{-t^2/2}{|J_0|Np(1-p) +Nt/3 }\right).
\end{equation*}
Taking $t=r(|J|N)^{1/2}$, using \eqref{eq: |J_0| upper bound}, \eqref{eq: lower bound for |J_0|} and assumption $|J|\ge 2r^2N/\delta^4$, we obtain
\begin{equation}\label{eq: sum S in J_0}
\Pr{\Big|\sum_{(i,j)\in J_0} S_{ij} - N|J|(1-w)p  \Big|  \le  2rN|J|^{1/2}} \ge 1-3\exp(-r^2).
\end{equation}

We are now ready to bound $p_0$ using \eqref{eq: first two way bounds for p_0} and subsequent estimates of terms in \eqref{eq: first two way bounds for p_0}. Using \eqref{eq: |J_0| upper bound}, \eqref{eq: upper bound of I cap J_1 complete}, \eqref{eq: upper bound of I tilde cap J_0 complete}, \eqref{eq: upper bound of I cap J_0 complete}, \eqref{eq: sum S in J_0} and assumptions on $N$, $|J|$ and $\gamma^2(\hat{c},c)$, we obtain
\begin{equation}\label{eq: p_0 bound complete}
  \Pr{ \max\left\{ \delta, h^{-1}\left(\frac{q}{2}+\frac{h(q)}{2}\right)\right\} \le p_0 \le \frac{1}{2} - \delta} \ge 1- 10\exp(-r^2).
\end{equation}

\medskip
\noindent{\em False negative estimate $q_0$.} Recall that
$$
q_0 = \frac{\sum_{(i,j)\in \hat{I}_c} (N-S_{ij}) }{N|\hat{I}_c|},
$$
where $\hat{I}_c$ is the set of all $(i,j)\in \hat{J}$ such that $S_{ij}\ge N/2$.
Note that
$$
\hat{I}_c =  \big(\hat{I}_c \cap J_1\big) \cup \big(\hat{I}_c\cap J_0\big) \cup \big(\hat{I}_c \setminus J\big).
$$
%We show that $q_0$ is mostly determined by $\tilde{I} \cap J_1$ while $\tilde{I}\cap J_0$ and $\tilde{I} \cap \tilde{J}$ are negligible.
Using the partition of $\hat{I}_c$, the bound $N-S_{ij}\le N$ and the definition of $\gamma^2(\hat{c},c)$, we obtain
\begin{equation}\label{eq: q_0 bounds first}
  \frac{\sum_{(i,j)\in J_1} (N-S_{ij})-N|I\cap J_1|-N\gamma^2(\hat{c},c)|J|}{N\left(|J_1|+|I_c\cap J_0|+\gamma^2(\hat{c},c)|J|\right)} \le
  q_0  \le \frac{\sum_{(i,j)\in J_1} (N-S_{ij}) +  N|I_c\cap J_0| + N\gamma^2(\hat{c},c)|J|}{N|I_c\cap J_1|-N\gamma^2(\hat{c},c)|J|}.
\end{equation}

All terms in \eqref{eq: q_0 bounds first} have been estimated except $\sum_{(i,j)\in J_1} (N-S_{ij})$.
Condition on $A$, by Bernstein's inequality, we have
\begin{equation*}
\Pr{\Big| \sum_{(i,j)\in J_1}  (N-S_{ij})- |J_1| N q \Big| > t} \le 2 \exp\left(\frac{-t^2/2}{|J_1|Nq(1-q) +Nt/3 }\right).
\end{equation*}
Taking $t=r(|J|N)^{1/2}$, using \eqref{eq: |J_1| upper bound}, \eqref{eq: lower bound for |J_1|} and $|J|\ge 2r^2N/\delta^4$, we obtain
\begin{equation}\label{eq: sum S in J_1}
\Pr{\Big|\sum_{(i,j)\in J_1}  (N-S_{ij}) - N|J|wq  \Big|  \le  2rN|J|^{1/2}} \ge 1-3\exp(-r^2).
\end{equation}

We are now ready to bound $q_0$ using \eqref{eq: q_0 bounds first} and subsequent estimates of terms in \eqref{eq: q_0 bounds first}.
Using \eqref{eq: |J_1| upper bound}, \eqref{eq: upper bound of I tilde cap J_0 complete}, \eqref{eq: upper bound of I tilde cap J_1 complete}, \eqref{eq: upper bound of I cap J_1 complete}, \eqref{eq: sum S in J_1} and assumptions on $N$, $|J|$ and $\gamma^2(\hat{c},c)$, we obtain
\begin{equation}\label{eq: bound on q_0 complete}
  \Pr{ \max\left\{ \delta, h^{-1}\left(\frac{q}{2}+\frac{h(q)}{2}\right)\right\} \le  q_0 \le \frac{1}{2} - \delta } \le 10\exp(-r^2).
\end{equation}

Finally, the claim of Lemma~\ref{lem: validity of initial estimates} follows from \eqref{eq: bound on w_0 final}, \eqref{eq: p_0 bound complete} and \eqref{eq: bound on q_0 complete}.
\end{proof}

%\subsection{Community detection error}
%Before proving our main theorem, we need a bound on community detection error $\gamma(\hat{c},c)$.
%Recall that $\hat{c}$ is the output of regularized spectral clustering with the input $\hat{A}$ defined by $\hat{A}_{ij}=\onevector(S_{ij}\ge N/2)$.
%
%\begin{lemma}\label{lem: community detection error}
%Let
%\end{lemma}
%\begin{proof}
%
%\end{proof}

\begin{proof}[Proof of Theorem~\ref{thm: guarantee of EM algorithm}]\label{pr: proof of main theorem}
  The proof of Theorem~\ref{thm: guarantee of EM algorithm} follows directly from Lemma~\ref{lem: contraction}, Corollary~\ref{cor: sample updates} and Lemma~\ref{lem: validity of initial estimates}.
\end{proof}

\section{Extension to Weighted Graphs} \label{ap: weighted}
In this section, we briefly present an extension of the model considered in
the main text to the case where both the
latent graph and the observed graphs may
have a broader class of edge noise distributions,
including the possibility of weighted edges.
As before, denote by $c \in [K]^n$
the vector of vertex community memberships,
and let $A \in \R^{n \times n}$ be the symmetric adjacency matrix of a
(possibly weighted) graph.
We assume that each of the $N$ networks
$\{ A^{(m)} \}_{m=1}^N$ has independent edges distributed as
\begin{equation*}
A^{(m)}_{ij} \sim f( A^{(m)}_{ij} \mid A_{ij} )
\end{equation*}
for some distribution $f$ parameterized by $A_{ij}$.
Suppose further that the upper-triangular entries of $A$
are themselves distributed according to vertex community memberships.
That is, there is a distribution $g$ parameterized by the entries of a
$K$-by-$K$ array $B$,
\begin{equation*}
A_{ij} \sim g( A_{ij} \mid B_{c_i c_j} ).
\end{equation*}
The setting of the main text is a special case, with $f$ a Bernoulli distribution with probability of success equal to $P_{ij} 1\{A_{ij} = 0\} + (1-Q_{ij})1\{A_{ij} = 1\}$, and $g$ corresponding to a stochastic block model with connectivity matrix $B$.   
%\begin{equation*} \begin{aligned}
%f( x \mid A_{ij} )
%	&= \begin{cases}
%		P_{ij} &\mbox{ if } x=1, A_{ij} = 0 \\
%		1-P_{ij} &\mbox{ if } x=0, A_{ij} = 0 \\
%		Q_{ij} &\mbox{ if } x=0, A_{ij} = 1 \\
%		1-Q_{ij} &\mbox{ if } x=0, A_{ij} = 1
%		\end{cases}, \\
%g( A_{ij} \mid B_{k,\ell} ) &= \begin{cases}
%				B_{k,\ell} &\mbox{ if } A_{ij} = 1 \\
%				1-B_{k,\ell} &\mbox{ if } A_{ij} = 0
%				\end{cases}
%\end{aligned} \end{equation*}	

We are interested in choices of $f$ and $g$ 
for which results akin to Theorem~\ref{thm: guarantee of EM algorithm} hold under reasonable assumptions.
Abstracting the algorithm from Section~\ref{subsec: em-type algorithm}
suggests the following procedure for estimating $B$, $A$ and $c$ 
from the observed graphs $\{ A^{(m)} \}_{m=1}^N$.
Suppose that we have initial estimates $\hat{A}$ and $\hat{c}$ for $A$ and $c$,
obtained, for example, by averaging the observed networks followed by
spectral clustering.
We then repeat the following two  steps for $T$ iterations:
\begin{enumerate}
\item Alternate the following until convergence:   
  \begin{enumerate}
  \item Estimate $\{ B_{k \ell} : 1 \le k \le \ell \le K \}$ by
\begin{equation} \label{eq:update:what}
\hat{w}_{k \ell} = \arg \max_w
                \prod_{i,j : \hat{c}_i=k, \hat{c}_j=\ell} g( \hat{A}_{ij} \mid w ).
\end{equation}
\item Estimate $\{ A_{ij} : 1 \le i < j \le n \}$ by
\begin{equation} \label{eq:update:Ahat}
\hat{A}_{ij} = \arg \max_\alpha g( \alpha \mid \hat{w}_{\hat{c}_i,\hat{c}_j} )
                        \prod_{m=1}^N f( A^{(m)}_{ij} \mid \alpha )
\end{equation}
\end{enumerate}
%\item Repeat the previous two steps until a stopping condition	holds (e.g., convergence or maximum iterations).
\item Update $\hat{c}$ via spectral clustering of $\hat{A}$.
\end{enumerate}
Note that we have assumed that the updates \eqref{eq:update:what} and~\eqref{eq:update:Ahat}
are computed exactly, as will be possible for certain parametric choices
of $f$ and $g$, but that in other cases it may suffice to simply approximate
this optimization to some degree of precision.

In order to obtain a result similar to 
Theorem~\ref{thm: guarantee of EM algorithm},
we need conditions on $f$ and $g$ that guarantee the following
informally stated properties:
\begin{enumerate}
\item Spectral clustering of $A$ recovers a 
	\label{item:spectral1}
	large fraction of the community labels.
\item The initial estimate $\hat{A}$ based on $\{ A^{(m)} \}_{m=1}^N$
	\label{item:spectral2}
	concentrates around the true adjacency matrix $A$ in spectral norm
	(which implies that spectral clustering of $\hat{A}$ is a reasonable
	approximation to the spectral clustering of $A$,
	and hence recovers a large enough fraction of the community labels).
\item The distribution $g(A_{ij} \mid B_{k \ell})$ is such that
	the estimate $\hat{w}_{k \ell}$ based on
	$\{ A_{ij} : \hat{c}_i=k, \hat{c}_j=\ell \}$
	concentrates about the true value $w_{k \ell}$
	when suitably many entries of $\hat{c}$ are correct.
	\label{item:estimate1}
\item The distribution $f( A^{(m)}_{ij} \mid A_{ij} )$ is such that
	the estimate $\hat{A}_{i j}$ of $A_{ij}$ based on the samples
	$\{ A^{(m)}_{ij}: m=1,2,\dots,N \}$ concentrates
	suitaby well about the true value $A_{ij}$.
	\label{item:estimate2}
\end{enumerate}
Note that we state these conditions under the assumption that community
memberships are estimated via spectral clustering, hence
the spectral norm concentration required in condition~\ref{item:spectral2}.
All four above conditions depend, ultimately,
on certain concentration inequalities holding,
which will depend upon the choice of model.
For example, the condition described in~\ref{item:spectral1}
and~\ref{item:spectral2} can be ensured
using machinery similar to that in \cite{Lyzinski&Sussman&Tang&Athreya&Priebe2014} and \cite{Levin&Athreya&Tang&Lyzinski&Priebe2017}
to guarantee that the eigenvalues and eigenvectors of
$\hat{A}$ and $A$ are suitably close, provided the entries
$(\hat{A}-A)_{ij}$ are zero-mean with suitably-bounded moments.
We leave a precise
statement of the most general analogue of 
Theorem~\ref{thm: guarantee of EM algorithm} for future work,
and sketch one natural approach to the general setting below.  

Suppose we take 
$f( A^{(m)}_{ij} \mid A_{ij} )$ to be an exponential family,
so that
\begin{equation} \label{eq:def:f}
  f( x \mid \alpha ) = h(x) \exp\{ \alpha x - \BB(\alpha) \},
\end{equation}
where $\alpha \in \R$ is a parameter,
$h : \R \rightarrow \R$ is the base measure of the exponential family
and $\BB$ is an appropriately-chosen log-partition function.
Then,  using basic properties of exponential families, it is easy to construct a distribution $g$
such that inference as in the Algorithm of
Section~\ref{subsec: em-type algorithm} is feasible.  
By Proposition 1.6.1 in \cite{Bickel&Doksum}, given the exponential
family in \eqref{eq:def:f}, we can find an exponential family with
parameter $w \in \R^2$ that is a conjugate prior to the distribution $f$,
with density of the form
\begin{equation} \label{eq:def:g}
g( \alpha \mid w ) = \exp\{ w^T(\alpha, \BB(\alpha)) - \ZZ (w) \},
\end{equation}
where $\ZZ$ is a log-partition function.
By basic properties of conjugate priors,
%both the posterior $\tilde{f}( \alpha \mid x ; w )
%\propto f( x \mid \alpha )g( \alpha \mid w)$ 
%and $g( \alpha \mid w )$ are in the same exponential family.
%That is,
the posterior
$\tilde{f}(\alpha \mid x; w) \propto f( x \mid \alpha )g( \alpha \mid w)$
is an exponential family of the same form as $g( \alpha \mid w)$.
This is particularly useful, since maximum likelihood estimates
of exponential families are typically easy to compute, 
and thus the updates in ~\eqref{eq:update:what}
and~\eqref{eq:update:Ahat} are feasible.
Further, a suitable choice of $f$ will ensure that the concentration
inequalities in the conditions above hold.   
For instance, it is straightforward to verify that choosing $f$ to have at most subgamma tails \cite{Boucheron&Lugosi&Massart2013,Tropp2012}
leads to an analogue of Theorem~\ref{thm: guarantee of EM algorithm}
for this broader class of models.

As an illustrative example, consider the case where
$f$ is the density of an exponential distribution with
scale parameter $A_{ij}$, and $g$ is the density of a gamma distribution with
shape parameter $\kappa > 0$ and scale parameter $\sigma > 0$.
Under this approach, the matrix $B$ in the main text becomes a
$K$-by-$K$-by-$2$ array of parameters, with 
$B_{k \ell} = (\kappa_{k \ell}, \sigma_{k \ell})^T \in \R^2$
for all $k,\ell \in [K]$.
Taking $c_i = k$ and $c_j=\ell$ for ease of notation, we have
\begin{equation*} \begin{aligned}
f( A^{(m)}_{ij} \mid A_{ij} )
	&= \frac{ \exp\{ -A^{(m)}_{ij} / A_{ij} \} }{  A_{ij} } \\
g( A_{ij} \mid \kappa_{k \ell}, \sigma_{k \ell} )
&= \frac{ A_{ij}^{\kappa_{k \ell} - 1} \exp\{ -A_{ij}/\sigma_{k \ell} \} }
	{ \sigma_{k \ell}^{\kappa_{k \ell} }
		\Gamma( \kappa_{k \ell}, \sigma_{k \ell} ) },
\end{aligned} \end{equation*}
where $\Gamma$ denotes the gamma function.
Since we have chosen $g$ to be the conjugate prior to $f$,
and because maximum-likelihood estimators for the exponential and gamma
distributions can be computed with relative ease
(via numerical methods in the case of the gamma distribution),
the maximization problems in~\eqref{eq:update:what}
and~\eqref{eq:update:Ahat} are feasible.

Under this model, we can also ensure that conditions~\ref{item:spectral1}
and~\ref{item:spectral2}, as follows.
Since the entries of the matrix $A$ have subgamma tails
\cite{Boucheron&Lugosi&Massart2013}, one can show using the
results in \cite{Tropp2012} that $A$ concentrates about its mean
$\mathbb{E} A_{ij} = \kappa_{c_i c_j} \sigma_{c_i c_j}$ in spectral norm.
Thus, $\| A - \mathbb{E} A \| = O( h(\kappa,\sigma) \sqrt{n} )$ for a suitable
function $h$.
Similarly, the mean of the entries $N^{-1} \sum_{m=1}^N A^{(m)}_{ij}$
concentrate about their expectation $\kappa_{c_i c_j} \sigma_{c_i c_j}$,
and the spectral norm error grows as
$\| N^{-1} \sum_{m=1}^N A^{(m)} - A \| = O( \sqrt{n/N} h(1,A) )$.
Under suitable assumptions on the growth rates of $N$ and $n$, the community
sizes, and the parameters $A,\kappa,\sigma$, the techniques from
\cite{Lyzinski&Sussman&Tang&Athreya&Priebe2014,Levin&Athreya&Tang&Lyzinski&Priebe2017}
can be used to turn these two spectral norm bounds into a guarantee that
an asymptotically vanishing fraction of the vertices are mislabeled,
thus ensuring conditions \ref{item:spectral1} and \ref{item:spectral2}.
The fact that $\{ A_{ij}^{(m)}: m=1,2,\dots,N \}$ are drawn
i.i.d.\ from an exponential distribution ensures that
the MLE in~\eqref{eq:update:Ahat} concentrates about $A_{ij}$
for all $i,j \in [n]$, as required by condition~\ref{item:estimate2}.
This fact, along with the fact that the initial
estimate $\hat{c}$ recovers most entries of $c$
(ensured by conditions~\ref{item:spectral1} and~\ref{item:spectral2})
similarly imply that the maximizer in~\eqref{eq:update:what} is close
to the true value of $B_{c_i c_j} = (\kappa_{c_i c_j}, \sigma_{c_i c_j})$,
thus guaranteeing condition~\ref{item:estimate1}.
This general analysis can, of course, be applied to any choice of exponential
family $f$ and conjugate prior $g$
so long as $f$ and $g$ have suitably light tails.

\bibliography{allref}
\bibliographystyle{abbrv}

\end{document}